\documentclass{lmcs}
\pdfoutput=1

\usepackage{lastpage}
\lmcsdoi{15}{1}{9}
\lmcsheading{}{\pageref{LastPage}}{}{}%
{Mar.~28,~2018}{Feb.~05,~2019}{}


\usepackage[noadjust]{cite}
\usepackage[utf8]{inputenc} 

\usepackage{graphicx}

\usepackage{amsmath}
\usepackage{amsthm}
\usepackage{amsfonts,amssymb}
\usepackage{array}
\usepackage{proof}

\usepackage[all]{xy}

\usepackage{url}
\usepackage{hyperref}

\usepackage{graphicx}
\usepackage{tikz}
\usepackage{tikz-cd}
\usetikzlibrary{arrows}
\usetikzlibrary{calc}
\usetikzlibrary{trees}
\usetikzlibrary{shapes}
\usetikzlibrary{decorations.markings}
\usetikzlibrary{positioning}

\tikzset{
    fromroot/.style={draw,postaction={decorate},
        decoration={markings,mark=at position .55 with {\arrow{>}}}},
    toroot/.style={draw,postaction={decorate},
        decoration={markings,mark=at position .55 with {\arrow{<}}}},
}

\definecolor{lightblue}{RGB}{173,216,230}
\definecolor{indianred}{RGB}{205,92,92}

\tikzstyle{var}=[circle,fill=black,draw=black,scale=0.1]
\tikzstyle{lam}=[circle,fill=lightblue,draw=black,scale=0.6]
\tikzstyle{app}=[circle,fill=indianred,draw=black,scale=0.6]

\newcommand*{\imgcenter}[1]{\begingroup\setbox0=\hbox{#1}\parbox{\wd0}{\box0}\endgroup}
\newcommand\defn[1]{{\bf #1}}

\newcommand\tseq[2]{#1 \longrightarrow #2}
\newcommand\fseq[2]{#1 \Longrightarrow #2}

\newcommand{\medblackbullet}{\,\begin{picture}(-1,1)(-1,-3)\circle*{3}\end{picture}\ }
\newcommand{\medwhitebullet}{\,\begin{picture}(-1,1)(-1,-3)\circle{3}\end{picture}\ }
\newcommand\mul{\mathbin{\medblackbullet}}
\newcommand\mulop{\mathbin{\medwhitebullet}}
\newcommand\mulrule{\mul}

\newcommand\Catalan[1]{C_{#1}}
\newcommand\Int[1]{\mathcal{I}_{#1}}

\newcommand\atm{\mathrm{atm}}
\newcommand\foc{\mathrm{foc}}

\newcommand\set[1]{\left\{\,#1\,\right\}}

\newcommand\toFsym{\phi}
\newcommand\toF[1]{\phi[#1]}
\newcommand\actF[2]{{#1}\circledast #2}
\newcommand\deriv[1]{\mathcal{#1}}
\newcommand\toCsym{\psi}
\newcommand\toC[1]{\psi[#1]}

\newcommand\orig{\mathrm{amb}}

\newcommand\sz[1]{\#(#1)}
\newcommand\frontier[1]{\partial #1}

\newcommand\irr{{\mathrm{irr}}}

\newcommand\Frm{\mathsf{Y}}
\newcommand\Ctx{\mathsf{F}(\mathsf{Y})}

\newcommand\Tam[1]{\Frm_{#1}}
\newcommand\TamC[1]{\Ctx_{[#1]}}

\newcommand\Comp[1]{\mathrm{O}_{[#1]}}
\newcommand\CompC[1]{\mathrm{O}_{#1}}

\begin{document}

\title{A sequent calculus for a semi-associative law}
\titlecomment{This article is an extended version of a paper presented at FSCD 2017 \cite{Z2017assoc}.
The most significant difference is the inclusion of the proof for the lattice property via the coherence theorem, which was raised as an open problem in the FSCD version.
Various aspects of the presentation are also improved, with additional discussion of recent related work as well as some historical remarks.
}

\author[N.~Zeilberger]{Noam Zeilberger}
\address{
School of Computer Science\\
University of Birmingham\\
UK}
\email{zeilbern@cs.bham.ac.uk}

\begin{abstract}
We introduce a sequent calculus with a simple restriction of Lambek's product rules that precisely captures the classical Tamari order, i.e., the partial order on fully-bracketed words (equivalently, binary trees) induced by a semi-associative law (equivalently, right rotation).
We establish a focusing property for this sequent calculus (a strengthening of cut-elimination), which yields the following coherence theorem: every valid entailment in the Tamari order has exactly one focused derivation.
We then describe two main applications of the coherence theorem, including: 1. A new proof of the lattice property for the Tamari order, and 2. A new proof of the Tutte--Chapoton formula for the number of intervals in the Tamari lattice $\Tam{n}$.
\end{abstract}

\maketitle


\section*{Introduction}
\label{sec:intro}

\subsection{The Tamari order, Tamari lattices and associahedra}
\label{sec:intro:intro}

Suppose you are given a pair of binary trees $A$ and $B$ and the following problem:
\emph{transform $A$ into $B$ using only right rotations.}
Recall that a right rotation is an operation acting locally on a pair of internal nodes of a binary tree, rearranging them like so:
$$
\imgcenter{\includegraphics[width=1.5cm]{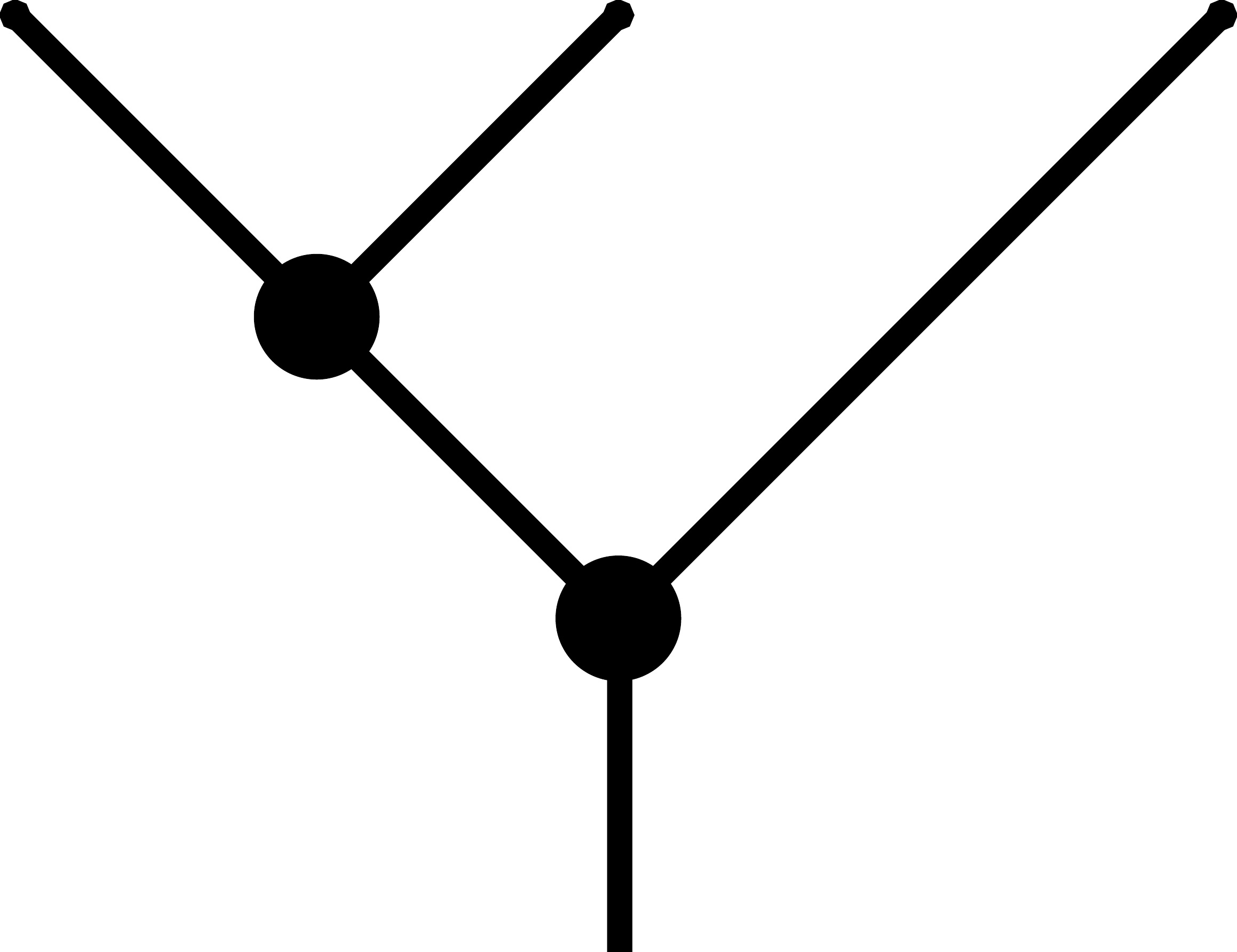}}
 \quad\longrightarrow\quad
\imgcenter{\includegraphics[width=1.5cm]{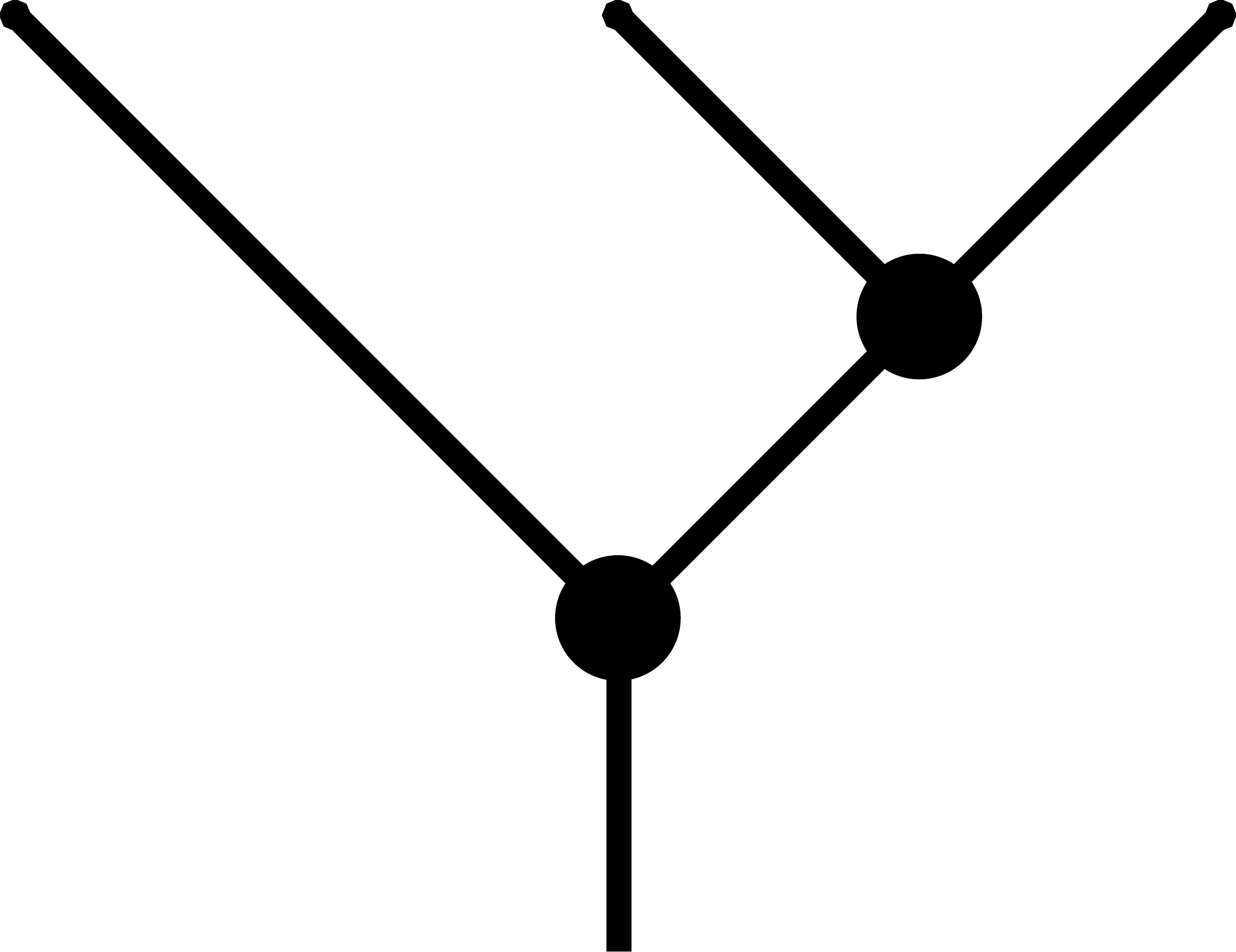}}
$$
Solving this problem amounts to showing that $A \le B$ in the \emph{Tamari order}.
Originally introduced by Dov Tamari for motivations in algebra \cite{Tamari1951phd}, this is the ordering on well-formed monomials (or ``fully-bracketed words'') induced by postulating that the multiplication operation $A\mul B$ obeys a \emph{semi-associative law}\footnote{Clearly, one has to make an arbitrary choice in orienting associativity from left-to-right or right-to-left.
The literature is inconsistent about this, but since the two possible orders are strictly dual it does not make much difference. \label{foot1}}
$$
(A\mul B)\mul C \le A\mul (B\mul C)
$$
and is monotonic in each argument:
$$
\infer{A_1\mul B_1 \le A_2\mul B_2}{A_1 \le A_2 & B_1 \le B_2}
$$
For example, the word $(p \mul  (q\mul r))\mul s$ is below the word $p \mul  (q \mul  (r \mul  s))$ in the Tamari order:
$$
\imgcenter{\includegraphics[width=1.5cm]{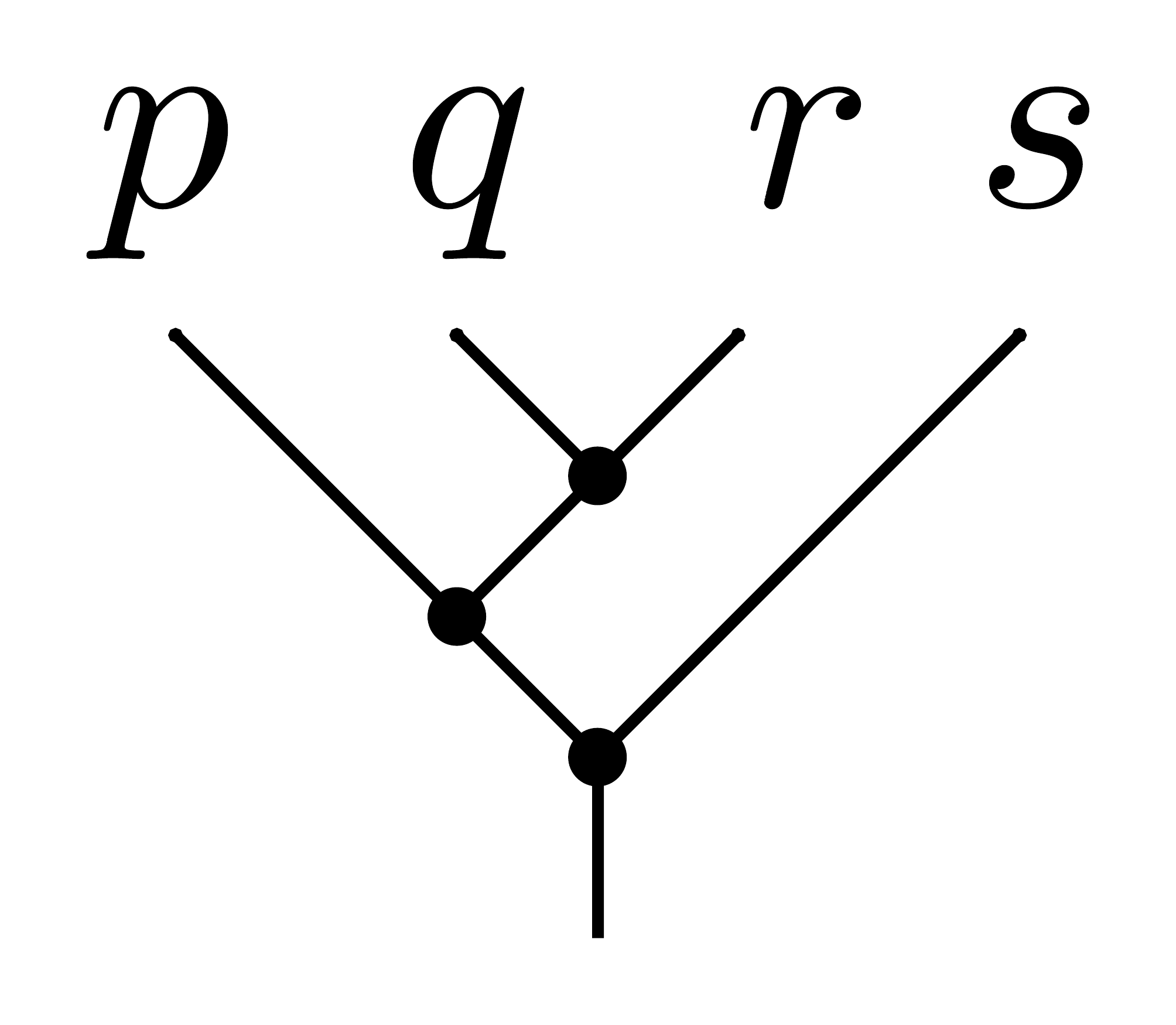}}
 \quad\longrightarrow\quad
\imgcenter{\includegraphics[width=1.5cm]{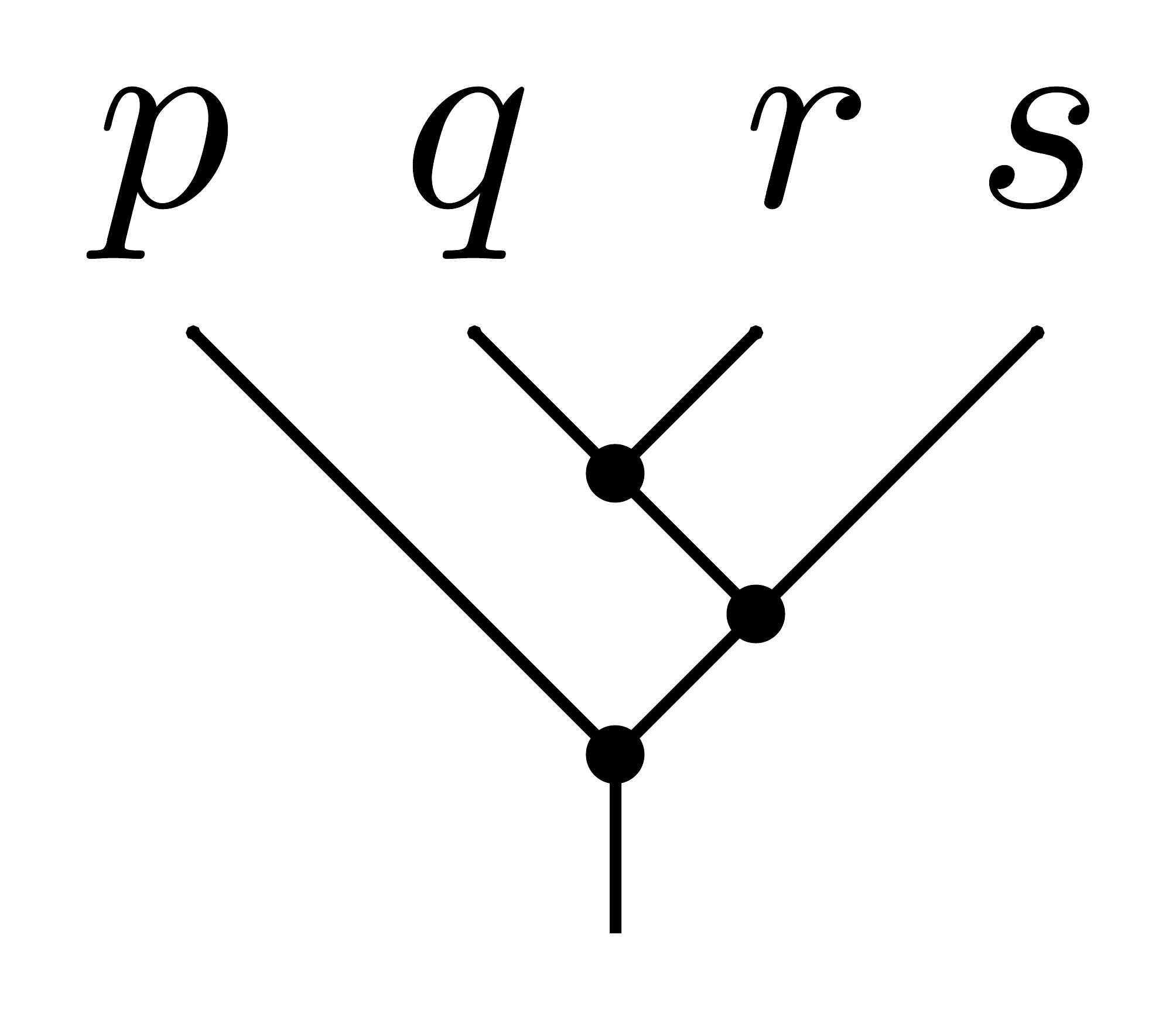}}
 \quad\longrightarrow\quad
\imgcenter{\includegraphics[width=1.5cm]{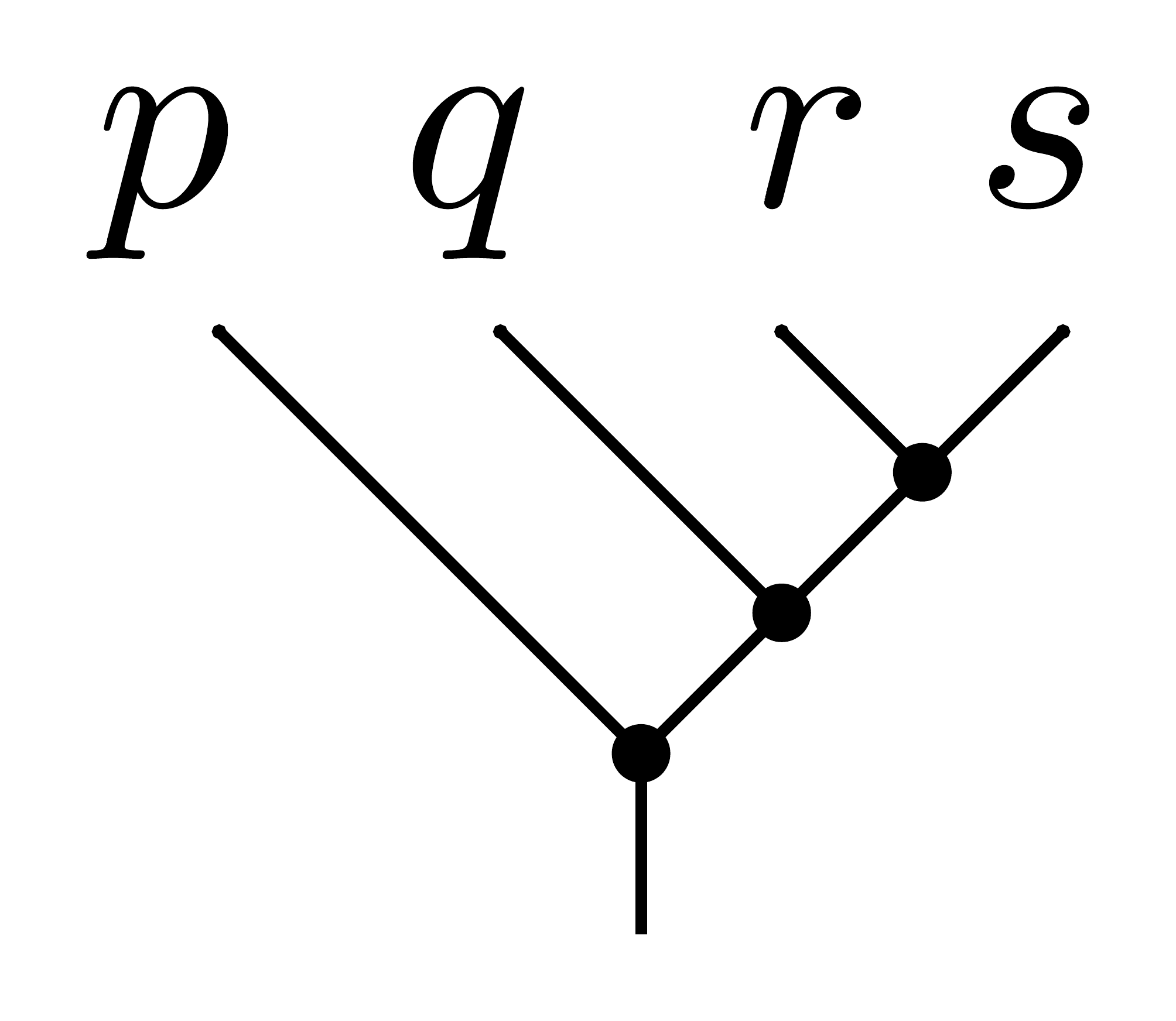}}
$$
The letters $p,q,\dots$ are just placeholders and what really matters is the underlying shape of the bracketings, which is what justifies the above description in terms of unlabelled binary trees.
Since such trees are enumerated by the ubiquitous Catalan numbers (there are $\Catalan{n} = \binom{2n}{n}/(n+1)$ distinct binary trees with $n$ internal nodes), which also count many other isomorphic families of objects, the Tamari order has many other equivalent formulations as well, such as on strings of balanced parentheses \cite{HuTa72}, triangulations of a polygon \cite{STT88,LackStreet2014},  or Dyck paths \cite{BeBo2009}.

For any fixed natural number $n$, the $\Catalan{n}$ unlabelled Catalan objects of that size form a lattice under the Tamari order, which is called the \emph{Tamari lattice} $\Tam{n}$.
For example, the left half of Figure~\ref{fig:tam3tam4} shows the Hasse diagram of $\Tam{3}$, which has the shape of a pentagon (readers familiar with category theory may recognize this as ``Mac~Lane's pentagon'' \cite{MacLane1963}).
\begin{figure}
\begin{center}
\includegraphics[width=6.5cm]{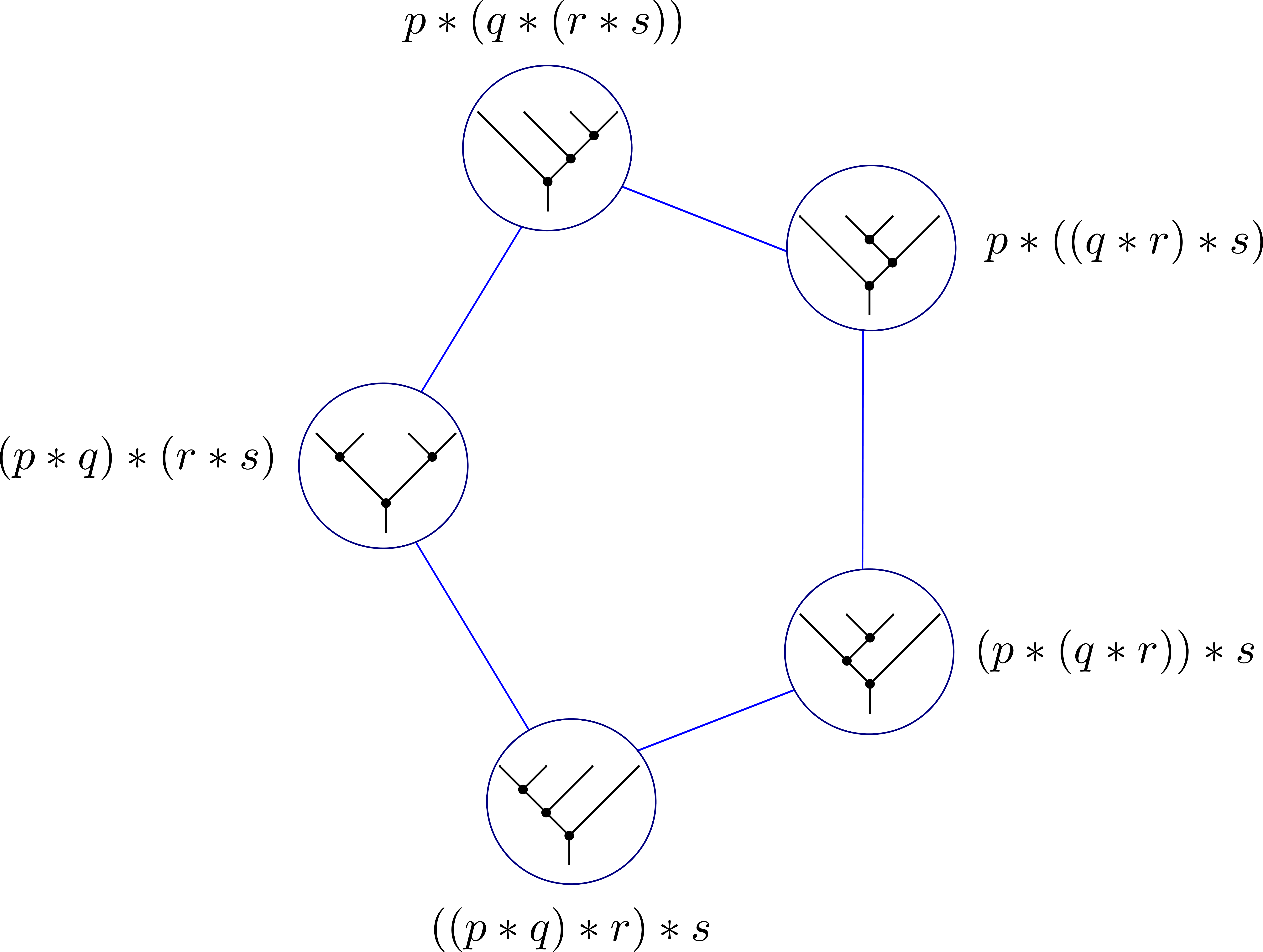}
\quad\vrule\quad
\includegraphics[width=6.5cm,angle=5]{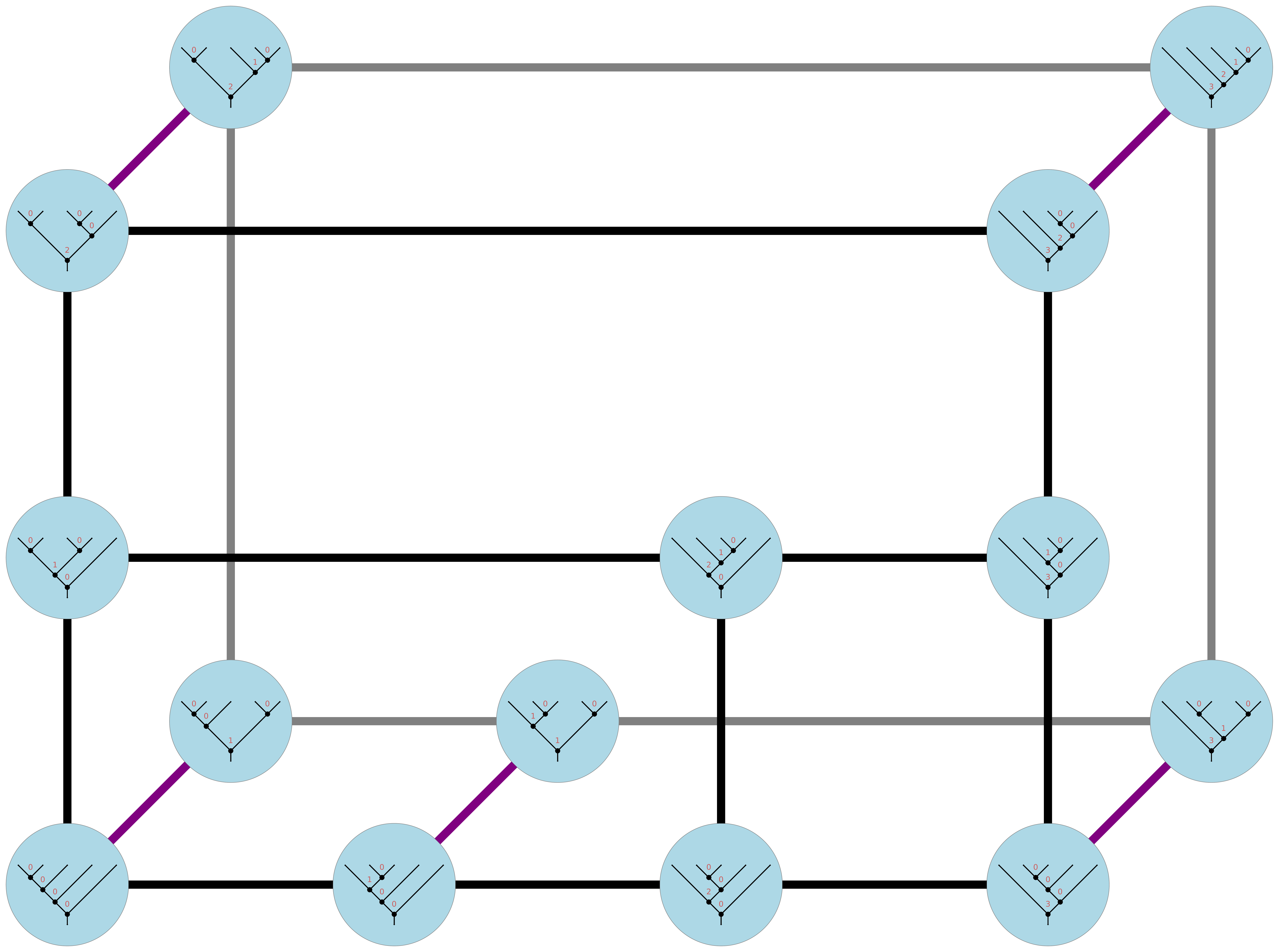} 
\caption{Hasse diagrams of the Tamari lattices $\Tam{3}$ and $\Tam{4}$.
The $\Tam{n}$ may be realized as convex polytopes: on the right is  a particular realization (due to Knuth \cite{Knuth1993TamariTalk}) that indicates the construction of meets, and hence the lattice structure of $\Tam{4}$.
}
\label{fig:tam3tam4}
\end{center}
\end{figure}
More generally, a fascinating property of the Tamari order is that each lattice $\Tam{n}$ generates via its Hasse diagram the underlying graph of an $(n-1)$-dimensional polytope known as an \emph{associahedron} \cite{Tamari1951phd,Stasheff1963,TamariFestschrift,CaspardSantocanaleWehrung2016}.

\subsection{A Lambekian analysis of the Tamari order}
\label{sec:intro:lambek}

We will introduce and study a surprisingly elementary presentation of the Tamari order as a \emph{sequent calculus} in the spirit of Lambek \cite{Lambek1958,Lambek1961}.
The calculus consists of just one \emph{left rule} and one \emph{right rule:}
$$
\infer[\mulrule L]{\tseq{A\mul B,\Delta}{C}}{\tseq{A,B,\Delta}{C}}
\qquad
\infer[\mulrule R]{\tseq{\Gamma,\Delta}{A\mul B}}{\tseq{\Gamma}{A} & \tseq{\Delta}{B}}
$$
together with two \emph{structural rules:}
$$
\infer[id]{\tseq{A}{A}}{}
\qquad
\infer[cut]{\tseq{\Gamma,\Theta,\Delta}{B}}{\tseq{\Theta}{A} & \tseq{\Gamma,A,\Delta}{B}}
$$
Here, the letters $A$, $B$, and $C$ again range over fully-bracketed words, although we will henceforth refer to them as logical ``formulas'', while the letters $\Gamma$, $\Delta$, and $\Theta$ range over lists of formulas called \emph{contexts}.

In fact, all of these rules come straight from Lambek \cite{Lambek1958}, except for the $\mulrule L$ rule which is a restriction of his left rule for products.
Lambek's original rule looked like this:
$$\infer[\mulrule L^\orig]{\tseq{\Gamma,A\mul B,\Delta}{C}}{\tseq{\Gamma,A,B,\Delta}{C}}$$
That is, Lambek's left rule allowed the formula $A\mul B$ to appear anywhere in the context, whereas our more restrictive rule $\mulrule L$ requires the formula to appear at the leftmost end of the context.
It turns out that this simple variation makes all the difference for capturing the Tamari order!

For example, here is a sequent derivation of the entailment
$(p \mul  (q\mul r))\mul s \le p \mul  (q \mul  (r \mul  s))$
(we write $L$ and $R$ as short for $\mulrule L$ and $\mulrule R$, and don't bother labelling instances of $id$):
$$
\infer[L]{\tseq{(p \mul  (q\mul r))\mul s}{p \mul  (q \mul  (r \mul  s))}}{
\infer[L]{\tseq{p \mul  (q\mul r),s}{p \mul  (q \mul  (r \mul  s))}}{
\infer[R]{\tseq{p, q\mul r,s}{p \mul  (q \mul  (r \mul  s))}}{
  \infer{\tseq{p}{p}}{} & 
  \infer[L]{\tseq{q\mul r,s}{q\mul (r\mul s)}}{
  \infer[R]{\tseq{q,r,s}{q\mul (r\mul s)}}{
    \infer{\tseq{q}{q}}{} & 
    \infer[R]{\tseq{r,s}{r\mul s}}{
        \infer{\tseq{r}{r}}{} &
        \infer{\tseq{s}{s}}{}}}}}}}
$$
If we had full access to Lambek's original rule then we could also derive the converse entailment (which is false for Tamari):
$$
\infer[L]{\tseq{p \mul  (q \mul  (r \mul  s))}{(p \mul  (q\mul r))\mul s}}{
\infer[L^\orig]{\tseq{p, q \mul  (r \mul  s)}{(p \mul  (q\mul r))\mul s}}{
\infer[L^\orig]{\tseq{p, q, r \mul  s}{(p \mul  (q\mul r))\mul s}}{
\infer[R]{\tseq{p, q, r,s}{(p \mul  (q\mul r))\mul s}}{
  \infer[R]{\tseq{p, q,r}{p \mul  (q\mul r)}}{
   \infer{\tseq{p}{p}}{} & 
    \infer[R]{\tseq{q,r}{q\mul r}}{
        \infer{\tseq{q}{q}}{} & 
        \infer{\tseq{r}{r}}{}
    }} &
  \infer{\tseq{s}{s}}{}
}}}}
$$
But with the more restrictive rule we can't -- the following soundness and completeness result will be established below.
\begin{clm}
\label{claim:soundcomplete}
$\tseq{A}{B}$ is derivable using the rules $\mulrule L$, $\mulrule R$, $id$, and $cut$ if and only if $A \le B$ holds in the Tamari order.
\end{clm}
\noindent
As Lambek emphasized, the real power of a sequent calculus comes when it is combined with Gentzen's \emph{cut-elimination} procedure \cite{Gentzen35}.
We will prove the following somewhat stronger form of cut-elimination:
\begin{clm}
\label{claim:focusing}
If $\tseq{\Gamma}{A}$ is derivable using the rules $\mulrule L$, $\mulrule R$, $id$, and $cut$, then it has a derivation using only $\mulrule L$ together with the following restricted forms of $\mulrule R$ and $id$: 
$$
\infer[\mulrule R^\foc]{\tseq{\Gamma^\irr,\Delta}{A\mul B}}{\tseq{\Gamma^\irr}{A} & \tseq{\Delta}{B}}
\qquad
\infer[id^\atm]{\tseq{p}{p}}{}
$$
where $\Gamma^\irr$ ranges over contexts that don't have a product $C\mul D$ at their leftmost end.
\end{clm}
\noindent
We refer to derivations constructed using only the rules $\mulrule L$, $\mulrule R^\foc$, and $id^\atm$ as \emph{focused derivations} (the above derivation of 
$(p \mul  (q\mul r))\mul s \le p \mul  (q \mul  (r \mul  s))$
is an example of a focused derivation)
because this stronger form of cut-elimination is precisely analogous to what in the literature on linear logic is called the \emph{focusing} property for a sequent calculus \cite{Andreoli92}.
Now, a careful analysis of the rules $\mulrule L$, $\mulrule R^\foc$, and $id^\atm$ 
also confirms that 
any sequent $\tseq{\Gamma}{A}$ has \emph{at most} one focused derivation.
Combining this with Claims~\ref{claim:soundcomplete} and \ref{claim:focusing}, we can therefore conclude that
\begin{clm}
Every valid entailment in the Tamari order has exactly one focused derivation.
\end{clm}
\noindent
We will see that this powerful \emph{coherence theorem} has several interesting applications.

\subsection{The surprising combinatorics of Tamari intervals, planar maps, and planar lambda terms}
\label{sec:intro:chapo}

The original impetus for this work came from wanting to better understand an apparent link between the Tamari order and lambda calculus, which was inferred indirectly via their mutual connection to the combinatorics of embedded graphs.

About a dozen years ago, Fréderic Chapoton \cite{Chapoton2006} proved the following surprising formula for the number of \emph{intervals} in the Tamari lattice $\Tam{n}$:
\begin{equation}
\frac{2(4n+1)!}{(n+1)!(3n+2)!} \label{tutte-formula1}
\end{equation}
Here, by an ``interval'' of a partially ordered set we just mean a valid entailment $A \le B$, which can also be thought of as the corresponding subposet of elements $[A,B] = \set{C \mid A \le C \le B}$.
For example, the lattice $\Tam{3}$ contains 13 intervals, as we can easily check:
\begin{center}
\includegraphics[width=6cm]{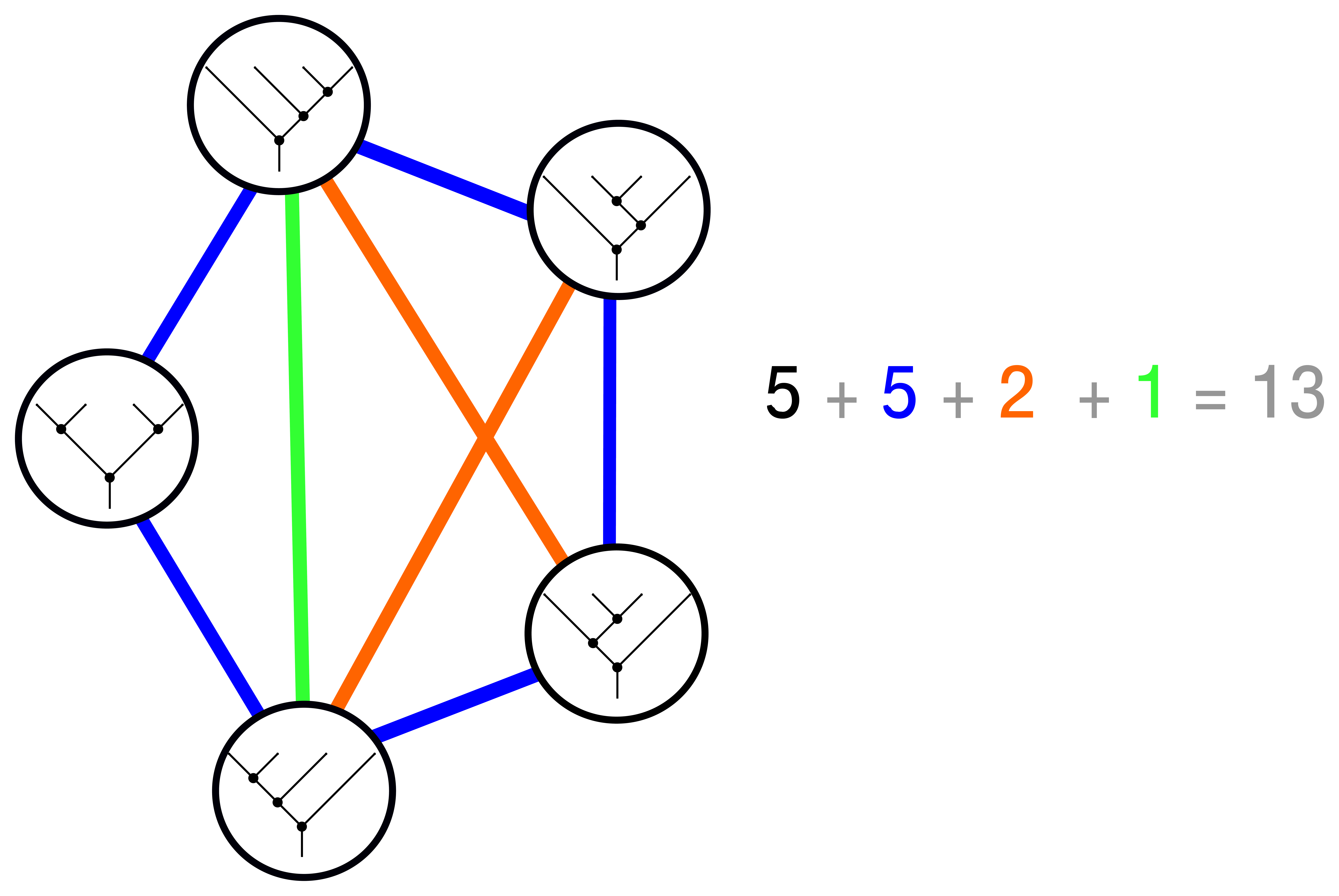}
\end{center}
And at the next dimension, $\Tam{4}$ contains exactly 68 intervals (exercise left to the reader!).

We will explain in Section~\ref{sec:enumeration} how the aforementioned coherence theorem can be used to give a new and somewhat more systematic proof of Chapoton's formula for the number of intervals in $\Tam{n}$.
Chapoton mentions, though, that the formula itself did not come out of thin air, but rather was found by querying the \emph{On-Line Encyclopedia of Integer Sequences} (OEIS) \cite{OEIS}.
Formula \eqref{tutte-formula1} is included within entry \href{https://oeis.org/A000260}{A000260} of the OEIS, and in fact it was derived over half a century ago by the graph theorist Bill Tutte for a seemingly unrelated family of objects: it counts the number of 3-connected, rooted planar triangulations\footnote{A \emph{rooted planar map} is a connected graph embedded in the 2-sphere (or, alternatively, the plane), with one oriented edge chosen as the root. 
A (rooted planar) \emph{triangulation} (dually, trivalent map) is a (rooted planar) map in which every face (dually, vertex) has degree three.
A map is said to be \emph{bridgeless} (respectively, \emph{3-connected}) if it has no edge (respectively, pair of vertices) whose removal disconnects the underlying graph. (Cf.~\cite{LZgraphs}.)} with $3(n+1)$ edges \cite{Tutte1962planartriangulations}.
Although Chapoton's proof of the correspondence was purely enumerative (going via generating functions), sparked by the observation, Bernardi and Bonichon \cite{BeBo2009} later found an explicit bijection between intervals of the Tamari order and triangulations.
Indeed, the same formula is also known to count bridgeless rooted planar maps with $n$ edges \cite{WalshLehman75}, and very recently, Fang \cite{Fang2018} has presented bijections between these three families of objects.

Quite independently, a few years ago the author noticed another surprising combinatorial connection between unconstrained (i.e., not necessarily bridgeless) rooted planar maps and a certain natural fragment of lambda calculus, consisting of terms which are $\beta$-normal and ``planar'' in the sense that variables are used exactly once and in (a well-defined) order.
Like in Chapoton's case, this link was found through the OEIS, since the corresponding counting sequence was already known 
(sequence \href{https://oeis.org/A000168}{A000168}, and once again it was first computed by Tutte, who derived another simple formula for the number of rooted planar maps with $n$ edges: $\frac{2(2n)!3^n}{n!(n+2)!}$).
This curious observation lead to a paper with Alain Giorgetti \cite{ZG2015corr}, where we gave an explicit (recursive) bijection between rooted planar maps and $\beta$-normal planar terms.
Later, in trying to better understand this connection and tie it to another independent connection found recently between linear lambda terms and (non-planar) trivalent maps \cite{BoGaJa2013}, the author noticed that the graph-theoretic condition of containing a bridge has a natural analogue in lambda calculus: it corresponds to the property of containing a \emph{closed subterm} \cite{Z2016trivalent}.
In particular, it is not difficult to check that the bijection described in \cite{ZG2015corr} restricts to a bijection between bridgeless rooted planar maps and $\beta$-normal planar terms with no closed subterms, and therefore that the same formula \eqref{tutte-formula1} also enumerates such terms by size.

Since both the Tamari lattice and lambda calculus are apparently rich with connections to the combinatorics of maps, it is natural to ask
whether they are also more directly related to each other.
This was the original motivation for our proof-theoretic analysis of the Tamari order, and an explicit bijection between Tamari intervals and $\beta$-normal planar terms with no closed subterms was given in an earlier version of this paper \cite{Z2017tamari}, relying in an essential way on the coherence theorem for its proof of correctness.
On the other hand, the meaningfulness of that bijection was not completely clear\footnote{The bijection in \cite{Z2017tamari} is simple and easy to verify, but the issue is that there are two different ordering conventions one can plausibly use to define the planarity of a lambda term (referred to as \emph{LR-planarity} and \emph{RL-planarity} in \cite{ZG2015corr}).
Although both conventions induce families of $\beta$-normal planar terms which are equinumerous with Tamari intervals, only one of the two conventions (namely, RL-planarity) is closed under $\beta$-reduction, and the bijection in \cite{Z2017tamari} is based on the one that isn't (namely, LR-planarity).}, and so it is omitted here in the hope that a more algebraic account will be forthcoming.
Even without this particular application, I believe that the sequent calculus is of intrinsic mathematical interest: Tamari lattices and related associahedra have been studied for over sixty years, so the fact that such an elementary and natural proof-theoretic characterization of semi-associativity has been seemingly overlooked is surprising.\footnote{And perhaps it is a bit of an accident of history: see Appendix~\ref{sec:postscript}!}
In addition to an alternate derivation of Chapoton's result, we will see that it leads to a new and more conceptual proof of the lattice property for $\Tam{n}$.

\subsection{Plan of the paper}

In Section~\ref{sec:seqcal} we establish all of the proof-theoretic properties claimed above, including soundness and completeness of the sequent calculus relative to the Tamari order, as well as the coherence theorem.
Then in Section~\ref{sec:lattice} we explain how the focusing property can be used to prove the lattice property for the Tamari order, and in Section~\ref{sec:enumeration} we explain how the coherence theorem can be used to derive the Tutte--Chapoton formula for the number of intervals in the Tamari lattice $\Tam{n}$.
Appendix~\ref{sec:postscript} includes some brief historical remarks about Tamari and Lambek's overlapping work, and a surprising but unfortunately missed connection.

\section{A sequent calculus for the Tamari order}
\label{sec:seqcal}

\subsection{Definitions and terminology}
\label{sec:seqcal:def}

For reference, we recall here the definition of the sequent calculus introduced in Section~\ref{sec:intro:lambek}, and clarify some notational conventions.
The four rules of the sequent calculus are:
$$
\infer[\mulrule L]{\tseq{A\mul B,\Delta}{C}}{\tseq{A,B,\Delta}{C}}
\qquad
\infer[\mulrule R]{\tseq{\Gamma,\Delta}{A\mul B}}{\tseq{\Gamma}{A} & \tseq{\Delta}{B}}
$$
$$
\qquad
\infer[id]{\tseq{A}{A}}{}
\qquad
\infer[cut]{\tseq{\Gamma,\Theta,\Delta}{B}}{\tseq{\Theta}{A} & \tseq{\Gamma,A,\Delta}{B}}
$$
Uppercase Latin letters ($A,B,\dots$) range over \defn{formulas}, corresponding to fully-bracketed words: 
we say that every formula is either \defn{atomic} (ranged over by lowercase Latin letters $p,q,\dots$) or else \defn{non-atomic} if it is a product ($A\mul B$).
Uppercase Greek letters $\Gamma,\Delta,\dots$ range over \defn{contexts}, corresponding to lists of formulas, with concatenation of contexts (which is of course a strictly associative operation) indicated by a comma.
We will often restrict to non-empty contexts (that is, to lists of length $\ge 1$), although not always (the empty context is denoted $\Delta = \cdot$).
Let us emphasize that as in Lambek's system \cite{Lambek1958} but in contrast to Gentzen's original sequent calculus \cite{Gentzen35}, there are no rules of \emph{weakening}, \emph{contraction}, or \emph{exchange}, so the order and the number of occurrences of a formula within a context matters.

A \defn{sequent} is a pair of a context $\Gamma$ and a formula $A$.
Abstractly, a \defn{derivation} is a tree (technically, a \emph{rooted planar tree with boundary}) 
whose internal nodes are labelled by the names of rules and whose edges are labelled by sequents satisfying the constraints indicated by the given rule.
The \defn{conclusion} of a derivation is the sequent labelling its outgoing root edge, while its \defn{premises} are the sequents labelling any incoming leaf edges.
A derivation with no premises is said to be \defn{closed}.
In addition to writing ``$\tseq{\Gamma}{A}$'' as a notation for sequents, sometimes we also use it as a shorthand to indicate that the given sequent is \defn{derivable} using the above rules, in other words that there exists a closed derivation whose conclusion is that sequent (it will always be clear which of these two senses we mean).
Sometimes we will need to give an explicit name to a derivation with a given conclusion, in which case we write the name over the sequent arrow.

As in the introduction,
when constructing derivations we sometimes write $L$ and $R$ as shorthand for $\mulrule L$ and $\mulrule R$, and usually don't bother labelling the instances of $id$ and $cut$ since they are always clear from context.

Finally, we introduce a few more specialized notions.
We define the \defn{frontier} $\frontier{A}$ of a formula $A$ to be the ordered list of atoms occurring in $A$ (with $\frontier{(A\mul B)} = (\frontier{A},\frontier{B})$ and $\frontier{p} = p$), and the frontier of a context $\Gamma = A_1,\dots,A_n$ as the concatenation of the frontiers of its formulas ($\frontier{\Gamma} = \frontier{A_1},\dots,\frontier{A_n}$).
If $\sigma$ is an arbitrary function sending atoms to atoms (not necessarily injectively), we write $\sigma A$ and $\sigma\Gamma$ for the evident actions on formulas and contexts defined by homomorphic extension.
The following properties are immediate by examination of the four sequent calculus rules.
\begin{prop}
\label{prop:frontier}
Suppose that $\tseq{\Gamma}{A}$. Then
\begin{enumerate}
\item (Frontier preservation:) $\frontier{\Gamma} = \frontier{A}$.
\item (Relabelling:) $\tseq{\sigma\Gamma}{\sigma A}$ for any function $\sigma$ sending atoms to atoms.
\end{enumerate}
\end{prop}

\subsection{Completeness}
\label{sec:seqcal:completeness}

Let $A \le B$ denote the Tamari order on formulas, that is, the least preorder such that $(A\mul B)\mul C \le A\mul (B\mul C)$ for all $A$, $B$, and $C$, and such that $A_1 \le A_2$ and $B_1 \le B_2$ implies $A_1\mul B_1 \le A_2\mul B_2$.

We begin by establishing completeness of the sequent calculus relative to the Tamari order, which is the easier direction.

\begin{thm}[Completeness]
\label{thm:completeness}
If $A \le B$ then $\tseq{A}{B}$.
\end{thm}
\begin{proof}
We must show that the relation $\tseq{A}{B}$ is reflexive and transitive, and that the multiplication operation satisfies a semi-associative law and is monotonic in each argument.
All of these properties are straightforward:
\begin{enumerate}
\item Reflexivity: immediate by $id$.
\item Transitivity: immediate by $cut$ (with $\Gamma = \Delta = \cdot$ and $\Theta$ a singleton).
\item Semi-associativity: 
$$
\infer[L]{\tseq{(A\mul B)\mul C}{A\mul (B\mul C)}}{
\infer[L]{\tseq{A\mul B,C}{A\mul (B\mul C)}}{
\infer[R]{\tseq{A,B,C}{A\mul (B\mul C)}}{
 \infer{\tseq{A}{A}}{} &
 \infer[R]{\tseq{B,C}{B\mul C}}{
   \infer{\tseq{B}{B}}{} & \infer{\tseq{C}{C}}{}}}}}
$$
\item Monotonicity:
  \[
\infer[L]{\tseq{A_1\mul B_1}{A_2\mul B_2}}{
\infer[R]{\tseq{A_1,B_1}{A_2\mul B_2}}{
 \tseq{A_1}{A_2} & \tseq{B_1}{B_2}}}
\tag*{\qedhere}
  \]
\end{enumerate}
\end{proof}

\subsection{Soundness}
\label{sec:seqcal:soundness}

To prove soundness relative to the Tamari order, first we have to explain the interpretation of general sequents.
The basic idea is that we can interpret a \emph{non-empty context} as a \emph{left-associated product}.
Thus, a general sequent of the form
$$
\tseq{A_0,A_1,\dots,A_n}{B}
$$
is interpreted as an entailment of the form
$$
(\cdots(A_0\mul A_1)\mul \cdots)\mul A_n \le B
$$
in the Tamari order.
In terms of binary trees, the sequent can be visualized like so:
$$
\vcenter{\hbox{
\begin{tikzpicture}[scale=0.6,thick, grow'=up,
    level 1/.style={level distance=1cm},
    level 2/.style={level distance=1cm,sibling distance=2cm},
    level 3/.style={level distance=1cm,sibling distance=2cm},
    level 4/.style={level distance=1cm,sibling distance=2cm}]
    \coordinate
        child { [fill] circle (4pt)
         child { [fill] circle (4pt)
            child { [fill] circle (4pt)
              child { node {$A_0$} }
              child { node (a2) {$A_1$} }
            }
            child { node (an-1) {$\color{white}A$} }
          }
          child { node (an) {$A_n$} }
        };
    \node at ($ (a2)!0.45!(an) + (0.1,0.1)$) {\rotatebox{0}{$\ddots$}}; 
\end{tikzpicture}}}
\quad\longrightarrow\quad
\vcenter{\hbox{
\begin{tikzpicture}[scale=0.6,thick, grow'=up,
    level 1/.style={level distance=1cm}]
    \coordinate child { node {$B$} };
\end{tikzpicture}}}
$$
That is, the context provides information about the \emph{left-branching spine} of the tree which is below in the right rotation order.

Let $\toF{-}$ be the operation taking any non-empty context $\Gamma$ to the formula $\toF{\Gamma}$ given by its left-associated product.
Recursively, $\toF{-}$ is defined by the following equations:
$$
\toF{A}
= A 
\qquad
\toF{\Gamma,A}
= \toF{\Gamma} \mul A
$$
Critical to soundness of the sequent calculus will be the following property of $\toF{-}$ (expressing that it is an ``oplax'' homomorphism, in the terminology of monoidal functors):
\begin{prop}[Oplaxity]
\label{prop:oplax}
$\toF{\Gamma,\Delta} \le \toF{\Gamma}\mul \toF{\Delta}$ for all non-empty contexts $\Gamma$ and $\Delta$.
\end{prop}
\begin{proof}
By induction on $\Delta$.
The case of a singleton context $\Delta = A$ is immediate.
Otherwise, if $\Delta = (\Delta',A)$, we reason as follows:
\begin{align*}
\toF{\Gamma,\Delta',A} &= \toF{\Gamma,\Delta'}\mul A &\text{(by definition)}
\\ &
\le (\toF{\Gamma}\mul \toF{\Delta'})\mul A & \text{(by the i.h.~and monotonicity)}
\\ &
\le \toF{\Gamma}\mul (\toF{\Delta'}\mul A) &\text{(semi-associativity)}
\\ &
 = \toF{\Gamma}\mul \toF{\Delta',A}  &\text{(by definition)}
\tag*{\qedhere}
\end{align*}
\end{proof}
\noindent
The operation $\toF{-}$ can also be equivalently described in terms of a \emph{right action} $\actF{A}{\Delta}$ of an arbitrary context on a formula, defined by the following equations:
$$
\actF{A}{\cdot}
= A
\qquad
\actF{A}{(\Delta,B)}
= (\actF{A}{\Delta})\mul B
$$
\begin{prop}[Right action]
\label{prop:equiv}
$\toF{\Gamma,\Delta} = \actF{\toF{\Gamma}}{\Delta}$ for all non-empty $\Gamma$ and arbitrary $\Delta$.
\end{prop}
\begin{prop}[Monotonicity]
\label{prop:monotonic}
If $A \le A'$ then $\actF{A}{\Delta} \le \actF{A'}{\Delta}$.
\end{prop}
\begin{proof}
Both properties are immediate by induction on $\Delta$.
(In the case of Prop.~\ref{prop:monotonic} we appeal to monotonicity of the operations $-\mul B$.)
\end{proof}

\noindent
We are now ready to prove soundness.
\begin{thm}[Soundness]
\label{thm:soundness}
If $\tseq{\Gamma}{A}$ then $\toF{\Gamma} \le A$.
\end{thm}
\begin{proof}
By induction on the (closed) derivation of $\tseq{\Gamma}{A}$.
There are four cases, corresponding to the four rules of the sequent calculus:
\begin{itemize}
\item (Case $\mulrule L$): The derivation ends in
$$\infer[\mulrule L]{\tseq{A\mul B,\Delta}{C}}{\tseq{A,B,\Delta}{C}}$$
Since $\toF{A\mul B} = A\mul B = \toF{A,B}$, we have that $\toF{A\mul B,\Delta} = \actF{\toF{A\mul B}}\Delta = \actF{\toF{A,B}}\Delta = \toF{A,B,\Delta} \le C$ by the induction hypothesis.

\item (Case $\mulrule R$): The derivation ends in
$$\infer[\mulrule R]{\tseq{\Gamma,\Delta}{A\mul B}}{\tseq{\Gamma}{A} & \tseq{\Delta}{B}}$$
By induction we have $\toF{\Gamma} \le A$ and $\toF{\Delta} \le B$.
But then $\toF{\Gamma,\Delta} \le \toF{\Gamma}\mul \toF{\Delta} \le A\mul B$
by oplaxity and monotonicity.
\item (Case $id$): Immediate by reflexivity.
\item (Case $cut$): The derivation ends in
$$
\infer[cut]{\tseq{\Gamma,\Theta,\Delta}{B}}{\tseq{\Theta}{A} & \tseq{\Gamma,A,\Delta}{B}}
$$
We reason as follows:
\begin{align*}
\toF{\Gamma,\Theta,\Delta} &
= \actF{\toF{\Gamma,\Theta}}{\Delta}  & \text{(right action)}
\\ &
\le \actF{(\toF{\Gamma}\mul \toF{\Theta})}{\Delta}     & \text{(oplaxity + monotonicity)}
\\ &
\le \actF{(\toF{\Gamma}\mul A)}{\Delta}                & \text{(i.h. + monotonicity)}
\\ & 
= \toF{\Gamma,A,\Delta}                            & \text{(right action)}
\\  &
\le B                                              & \text{(i.h.)}
\tag*{\qedhere}
\end{align*}
\end{itemize}
\end{proof}

\subsection{Focusing completeness}
\label{sec:seqcal:foccomp}

Cut-elimination theorems are a staple of proof theory, and often provide a rich source of information about a given logic.
In this section we will prove a \emph{focusing} property for the sequent calculus, which is an even stronger form of cut-elimination originally formulated by Andreoli in the setting of linear logic \cite{Andreoli92}.

\begin{defi}\label{def:focusing}
A context $\Gamma$ is said to be \defn{reducible} if its leftmost formula is non-atomic, and \defn{irreducible} otherwise.
A sequent $\tseq{\Gamma}{A}$ is said to be:
\begin{itemize}
\item \defn{left-inverting} if $\Gamma$ is reducible;
\item \defn{right-focusing} if $\Gamma$ is irreducible and $A$ is non-atomic;
\item \defn{atomic} if $\Gamma$ is irreducible and $A$ is atomic.
\end{itemize}
\end{defi}
\begin{prop}\label{prop:seqclass}
Any sequent is either left-inverting, right-focusing, or atomic.
\end{prop}
\begin{defi}
A closed derivation $\mathcal{D}$ is said to be \defn{focused} if left-inverting sequents only appear as the conclusions of $\mulrule L$, right-focusing sequents only as the conclusions of $\mulrule R$, and atomic sequents only as the conclusions of $id$.
\end{defi}
\noindent
We write ``$\Gamma^\irr$'' to indicate that a context $\Gamma$ is irreducible.
\begin{prop}\label{prop:equivfocus}
A closed derivation is focused if and only if it is constructed using only $\mulrule L$ and the following restricted forms of $\mulrule R$ and $id$ (and no instances of $cut$):
$$
\infer[\mulrule R^\foc]{\tseq{\Gamma^\irr,\Delta}{A\mul B}}{\tseq{\Gamma^\irr}{A} & \tseq{\Delta}{B}}
\qquad
\infer[id^\atm]{\tseq{p}{p}}{}
$$
\end{prop}
\noindent
\begin{exa}
One way to derive 
$\tseq{((p\mul q)\mul r)\mul s}{p\mul ((q\mul r)\mul s)}$
is by cutting together the two derivations
$$
\vcenter{ \infer[L]{\tseq{((p\mul q)\mul r)\mul s}{(p\mul (q\mul r))\mul s}}{
  \infer[R]{\tseq{(p\mul q)\mul r,s}{(p\mul (q\mul r))\mul s}}{
   \deduce{\tseq{(p\mul q)\mul r}{p\mul (q\mul r)}}{\mathcal{SA}_{p,q,r}} &
   \infer{\tseq{s}{s}}{}}}}
\qquad\text{and}\qquad
\vcenter{ \deduce{\tseq{(p\mul (q\mul r))\mul s}{p\mul ((q\mul r)\mul s)}}{\mathcal{SA}_{p,q\mul r,s}}}
$$
where $\mathcal{SA}_{A,B,C}$ is the derivation of the semi-associative law $\tseq{(A\mul B)\mul C}{A\mul (B\mul C)}$ from the proof of Theorem~\ref{thm:completeness}.
Clearly this is \emph{not} a focused derivation (besides the cut rule, it also uses instances of $\mulrule R$ and $id$ with a left-inverting conclusion).
However, it is possible to give a focused derivation of the same sequent:
$$
\infer[L]{\tseq{((p\mul q)\mul r)\mul s}{p\mul ((q\mul r)\mul s)}}{
\infer[L]{\tseq{(p\mul q)\mul r,s}{p\mul ((q\mul r)\mul s)}}{
\infer[L]{\tseq{p\mul q,r,s}{p\mul ((q\mul r)\mul s)}}{
\infer[R]{\tseq{p,q,r,s}{p\mul ((q\mul r)\mul s)}}{
  \infer{\tseq{p}{p}}{} &
  \infer[R]{\tseq{q,r,s}{(q\mul r)\mul s}}{
   \infer[R]{\tseq{q,r}{q\mul r}}{
     \infer{\tseq{q}{q}}{} &
     \infer{\tseq{r}{r}}{}} &
  \infer{\tseq{s}{s}}{}}}}}}
$$
\end{exa}
\noindent
In the below, we write ``$\fseq{\Gamma}{A}$'' as a shorthand notation to indicate that $\tseq{\Gamma}{A}$ has a (closed) focused derivation, and ``$\deriv{D} : \fseq{A}{B}$'' to indicate that $\deriv{D}$ is a particular focused derivation of $\tseq{A}{B}$.
\begin{thm}[Focusing completeness]
\label{thm:foc-completeness}
If $\tseq{\Gamma}{A}$ then $\fseq{\Gamma}{A}$.
\end{thm}
\noindent
To prove the focusing completeness theorem, it suffices to show that the cut rule as well as the unrestricted forms of $id$ and $\mulrule R$ are all \emph{admissible} for focused derivations, in the standard proof-theoretic sense that given focused derivations of their premises, we can obtain a focused derivation of their conclusion.
We begin by proving a \emph{focused deduction} lemma (cf.~\cite{ReedPfenning10}), which entails the admissibility of $id$, then show $cut$ and $\mulrule R$ in turn.
\begin{lem}[Deduction]
\label{lem:focgenid}
If $\fseq{\Gamma^\irr}{A}$ implies $\fseq{\Gamma^\irr,\Delta}{B}$ for all $\Gamma^\irr$, then $\fseq{A,\Delta}{B}$.
In particular, $\fseq{A}{A}$.
\end{lem}
\begin{proof}
By induction on the formula $A$:
\begin{itemize}
\item (Case $A = p$): Immediate by assumption, taking $\Gamma^\irr = p$ and $\fseq{p}{p}$ derived by $id^\atm$.
\item (Case $A = A_1\mul A_2$): 
We assume that $\fseq{\Gamma^\irr}{A_1\mul A_2}$ implies $\fseq{\Gamma^\irr,\Delta}{B}$ for all $\Gamma^\irr$, and we need to show that $\fseq{A_1\mul A_2,\Delta}{B}$.
By composing with the $\mulrule L$ rule,
$$
\infer[\mulrule L]{\tseq{A_1\mul A_2,\Delta}{B}}{\tseq{A_1,A_2,\Delta}{B}}
$$
we reduce the problem to showing $\fseq{A_1,A_2,\Delta}{B}$, and by the i.h. on $A_1$ it suffices to show that $\fseq{\Gamma_1^\irr}{A_1}$ implies $\fseq{\Gamma^\irr_1,A_2,\Delta}{B}$ for all contexts $\Gamma^\irr_1$.
Let $\deriv{D}_1 : \fseq{\Gamma^\irr_1}{A_1}$.
We can derive $\fseq{\Gamma^\irr_1,A_2}{A_1\mul A_2}$ by
$$
\deriv{D} \quad=\quad
\infer[\mulrule R]{\tseq{\Gamma^\irr_1,A_2}{A_1\mul A_2}}{
 \deduce{\tseq{\Gamma^\irr_1}{A_1}}{\deriv{D}_1} & \deduce{\tseq{A_2}{A_2}}{\deriv{D}_2}}
$$
where we apply the i.h. on $A_2$ to obtain $\deriv{D}_2$.
Finally, applying the assumption 
to $\deriv{D}$ (with $\Gamma^\irr = \Gamma^\irr_1,A_2$) we obtain the desired derivation of $\fseq{\Gamma_1^\irr,A_2,\Delta}{B}$.
\qedhere
\end{itemize}
\end{proof}
\noindent
\begin{lem}[Cut]
\label{lem:cut}
If $\fseq{\Theta}{A}$ and $\fseq{\Gamma,A,\Delta}{B}$ then $\fseq{\Gamma,\Theta,\Delta}{B}$.
\end{lem}

\begin{proof}
Let $\deriv{D} : \fseq{\Theta}{A}$ and $\deriv{E} : \fseq{\Gamma,A,\Delta}{B}$.
We proceed by a lexicographic induction, first on the cut formula $A$ and then on the pair of derivations $(\deriv{D},\deriv{E})$ (i.e., at each inductive step of the proof, either $A$ gets smaller, or it stays the same as one of $\deriv{D}$ or $\deriv{E}$ gets smaller while the other stays the same, cf.~\cite{Pfenning2000}).

In the case that $A = p$ we can apply the frontier preservation property (Prop.~\ref{prop:frontier}) to deduce that $\Theta = p$, so the cut is trivial and we just reuse the derivation $\deriv{E} : \fseq{\Gamma,p,\Delta}{B}$.
Otherwise we have $A = A_1\mul A_2$ for some $A_1,A_2$, and we proceed by case-analyzing the root rule of $\deriv{E}$:
\begin{itemize}
\item (Case $id^\atm$): Impossible since $A$ is non-atomic.
\item (Case $\mulrule R^\foc$): This case splits in two possibilities:
\begin{enumerate}
\item $\exists \Delta_1,\Delta_2$ such that $\Delta = \Delta_1,\Delta_2$ and
$$
\deriv{E}\quad=\quad\infer[\mulrule R]{\tseq{\Gamma^\irr,A,\Delta_1,\Delta_2}{B_1\mul B_2}}{\deduce{\tseq{\Gamma^\irr,A,\Delta_1}{B_1}}{\deriv{E}_1} & \deduce{\tseq{\Delta_2}{B_2}}{\deriv{E}_2}}
$$
\item $\exists \Gamma^\irr_1,\Gamma_2$ such that $\Gamma^\irr =\Gamma^\irr_1,\Gamma_2$ and
$$
\deriv{E}\quad=\quad\infer[\mulrule R]{\tseq{\Gamma^\irr_1,\Gamma_2,A,\Delta}{B_1\mul B_2}}{\deduce{\tseq{\Gamma^\irr_1}{B_1}}{\deriv{E}_1} & \deduce{\tseq{\Gamma_2,A,\Delta}{B_2}}{\deriv{E}_2}}
$$
\end{enumerate}
In the first case, we cut $\deriv{D}$ with $\deriv{E}_1$ to obtain $\fseq{\Gamma^\irr,\Theta,\Delta_1}{B_1}$, then recombine that with $\deriv{E}_2$ using $\mulrule R^\foc$ to obtain $\fseq{\Gamma^\irr,\Theta,\Delta_1,\Delta_2}{B_1\mul B_2}$.
The second case is similar.

\item (Case $\mulrule L$): This case splits into two possibilities:
\begin{enumerate}
\item $\exists C_1,C_2,\Gamma'$ such that $\Gamma = C_1\mul C_2,\Gamma'$ and
$$\deriv{E}\quad=\quad\infer[\mulrule L]{\tseq{C_1\mul C_2,\Gamma',A,\Delta}{B}}{\deduce{\tseq{C_1,C_2,\Gamma',A,\Delta}{B}}{\deriv{E}'}}$$
We cut $\deriv{D}$ into $\deriv{E}'$ and reapply the $\mulrule L$ rule.
\item $\Gamma = \cdot$ and
$$\deriv{E}\quad=\quad\infer[\mulrule L]{\tseq{A_1\mul A_2,\Delta}{B}}{\deduce{\tseq{A_1,A_2,\Delta}{B}}{\deriv{E}'}}$$
Now we need to analyze the root rule of $\deriv{D}$:
\begin{itemize}
\item (Case $\mulrule L$): $\exists C_1,C_2,\Theta'$ s.t. $\Theta = C_1\mul C_2,\Theta'$ and
$$
\deriv{D}\quad=\quad\infer[\mulrule L]{\tseq{C_1\mul C_2,\Theta'}{A_1\mul A_2}}{\deduce{\tseq{C_1,C_2,\Theta'}{A_1\mul A_2}}{\deriv{D}'}}
$$
We cut $\deriv{D}'$ into $\deriv{E}$ and reapply the $\mulrule L$ rule.
\item (Case $id^\atm$): Impossible.
\item (Case $\mulrule R^\foc$): $\exists \Theta^\irr_1,\Theta_2$ s.t. $\Theta^\irr = \Theta^\irr_1,\Theta_2$ and
$$
\deriv{D}\quad=\quad\infer[\mulrule R]{\tseq{\Theta^\irr_1,\Theta_2}{A_1\mul A_2}}{\deduce{\tseq{\Theta^\irr_1}{A_1}}{\deriv{D}_1} & \deduce{\tseq{\Theta_2}{A_2}}{\deriv{D}_2}}
$$
We cut both $\deriv{D}_1$ and $\deriv{D}_2$ into $\deriv{E}'$ (the cuts are at smaller formulas so the order doesn't matter).
\qedhere
\end{itemize}
\end{enumerate}

\end{itemize}
\end{proof}
\noindent
\begin{lem}[$\mulrule R$ admissible]
\label{lem:*r}
If $\fseq{\Gamma}{A}$ and $\fseq{\Delta}{B}$ then $\fseq{\Gamma,\Delta}{A\mul B}$.
\end{lem}
\begin{proof}
Let $\deriv{D} : \fseq{\Gamma}{A}$ and $\deriv{E} : \fseq{\Delta}{B}$.
We proceed by induction on $\deriv{D}$.
If $\Gamma$ is irreducible then we directly apply $\mulrule R^\foc$.
Otherwise, $\deriv{D}$ must be of the form 
$$
\deriv{D}\quad=\quad\infer[\mulrule L]{\tseq{C_1\mul C_2,\Gamma'}{A}}{\deduce{\tseq{C_1,C_2,\Gamma'}{A}}{\deriv{D}'}}
$$
so we apply the i.h. to $\deriv{D}'$ with $\deriv{E}$, and reapply the $\mulrule L$ rule.
\end{proof}
\begin{proof}[Proof of Theorem~\ref{thm:foc-completeness}]
An arbitrary closed derivation can be turned into a focused one by starting at the top of the derivation tree and using the above lemmas
to interpret any instance of $cut$ and of the unrestricted forms of $id$ and $\mulrule R$.
\end{proof}
\noindent
We mention here one simple application of the focusing completeness theorem.
\begin{prop}[Frontier invariance]
\label{prop:frontier-inv}
Let $\sigma$ be any function sending atoms to atoms.
Then $\tseq{\Gamma}{A}$ if and only if $\frontier{\Gamma} = \frontier{A}$ and $\tseq{\sigma\Gamma}{\sigma A}$.
\end{prop}
\begin{proof}
The forward direction is Prop.~\ref{prop:frontier}(1).
For the backward direction we use induction on focused derivations, which is justified by Theorem~\ref{thm:foc-completeness}.
The only interesting case is $\mulrule R^\foc$, where we can assume $\frontier{(\Gamma,\Delta)} = \frontier{(A\mul B)}$ and $\fseq{\sigma\Gamma}{\sigma A}$ ($\sigma\Gamma$ irreducible) and $\fseq{\sigma\Delta}{\sigma B}$.
By Prop.~\ref{prop:frontier}(1) we have $\frontier{\sigma\Gamma} = \frontier{\sigma A}$ and $\frontier{\sigma\Delta} = \frontier{\sigma B}$, but then elementary properties of lists imply that $\frontier{\Gamma} = \frontier{A}$ and $\frontier{\Delta} = \frontier{B}$, from which $\fseq{\Gamma}{A}$ ($\Gamma$ irreducible) and $\fseq{\Delta}{B}$ follow by the induction hypothesis, hence $\fseq{\Gamma,\Delta}{A\mul B}$.
\end{proof}
\noindent
If we let $\sigma = \textunderscore \mapsto p$ be any constant relabelling function, then speaking in the language of the introduction, Proposition~\ref{prop:frontier-inv} says that to check that two fully-bracketed words (a.k.a., formulas) are related, it suffices to check that their frontiers are equal and that the unlabelled binary trees describing their underlying multiplicative structure are related.
In other words, the restriction of the Tamari order to formulas $A$ such that $\frontier{A} = \Omega$, for any fixed frontier $\Omega$ of length $n+1$, defines a poset which is isomorphic to the poset $\Tam{n}$ of unlabelled binary trees under the right rotation order (remember that a binary tree with $n$ internal nodes has $n+1$ leaves).
Although this fact is intuitively obvious, trying to prove it directly by induction on general derivations in the sequent calculus fails, because in the case of the $cut$ rule we cannot assume anything about the frontier of the cut formula $A$.

\subsection{The coherence theorem}
\label{sec:seqcal:coherence}

We now come to our main result justifying the definition of the sequent calculus:
\begin{thm}[Coherence]
\label{thm:coherence}
Every derivable sequent has exactly one focused derivation.
\end{thm}
\noindent
Coherence is a direct consequence of focusing completeness and the following lemma:
\begin{lem}
\label{lem:atmost1}
For any context $\Gamma$ and formula $A$, there is at most one focused derivation of $\tseq{\Gamma}{A}$.
\end{lem}
\begin{proof}
We proceed by a well-founded induction on sequents, which can be reduced to multiset induction as follows.
Define the \defn{size} $\sz{A}$ of a formula $A$ by
$\sz{A\mul B} = 1+\sz{A}+\sz{B}$
and $\sz{p} = 0$.
(That is, $\sz{A}$ counts the number of multiplication operations occurring in $A$.)
Then any sequent $\tseq{A_1,\dots,A_n}{B}$ induces a multiset of sizes
$\left(\biguplus_{i=1}^n\sz{A_i}\right) \uplus \sz{B}$, 
and at each step of our induction this multiset will decrease in the multiset ordering.
There are three cases:
\begin{itemize}
\item (A left-inverting sequent $\tseq{A\mul B,\Delta}{C}$): Any focused derivation must end in $\mulrule L$, so we apply the i.h. to $\tseq{A,B,\Delta}{C}$.
\item (A right-focusing sequent $\tseq{\Gamma^\irr}{A\mul B}$): Any focused derivation must end in $\mulrule R$, and decide some splitting of the context into contiguous pieces $\Gamma_1^\irr$ and $\Delta_2$.
However, $\Gamma_1^\irr$ and $\Delta_2$ are uniquely determined by frontier preservation ($\frontier{\Gamma_1^\irr} = \frontier{A}$ and $\frontier{\Delta_2} = \frontier{B}$) and the equation $\Gamma^\irr = \Gamma^\irr_1,\Delta_2$.
So, we apply the i.h. to $\tseq{\Gamma_1^\irr}{A}$ and $\tseq{\Delta_2}{B}$.
\item (An atomic sequent $\tseq{\Gamma^\irr}{p}$): The sequent has one focused derivation if $\Gamma^\irr = p$, and otherwise it has none.
  \qedhere
\end{itemize}
\end{proof}
\begin{proof}[Proof of Theorem~\ref{thm:coherence}]
By Theorem~\ref{thm:foc-completeness} and Lemma~\ref{lem:atmost1}.
\end{proof}

\subsection{Notes and related work}
\label{sec:seqcal:notes}

The coherence theorem says in a sense that focused derivations provide a canonical representation for intervals of the Tamari order.
Although the representations are quite different, in this respect it is very roughly comparable to the ``unicity of normal chains'' that was established by Tamari and Friedman as part of their original proof of the lattice property of $\Tam{n}$ \cite{FrTa67}.
A natural question is whether the sequent calculus can be used to better understand and further simplify the proof of the lattice property: we turn to this question in Section~\ref{sec:lattice} below.

An easy observation is that one obtains the dual Tamari order (cf.~Footnote~\ref{foot1}) via a dual restriction of Lambek's original rules:
$$
\infer[\mulop L]{\tseq{\Gamma,A\mulop B}{C}}{\tseq{\Gamma,A,B}{C}} \qquad
\infer[\mulop R]{\tseq{\Gamma,\Delta}{A\mulop B}}{\tseq{\Gamma}{A} & \tseq{\Delta}{B}}
$$
These two forms of biased products cannot be combined naively, however, or one obtains a system without cut-elimination, as the following derivation shows:
$$
\infer[cut]{\tseq{p\mulop q,r}{p\mul (q\mul r)}}{
\infer[\mulop L]{\tseq{p\mulop q}{p\mul q}}{\infer[\mulrule R]{\tseq{p,q}{p\mul q}}{\infer{\tseq{p}{p}}{} & \infer{\tseq{q}{q}}{}}} &
\infer[\mulrule L]{\tseq{p\mul q,r}{p\mul (q\mul r)}}{
 \deduce{\tseq{p,q,r}{p\mul (q\mul r)}}{\vdots}}}
$$
On the other hand, the \emph{left}-biased product $A\mul B$ can be naturally combined with Lambek's \emph{right} division operation $B/A$,
$$
\infer{\tseq{B/A,\Gamma,\Delta}{C}}{\tseq{\Gamma}{A} & \tseq{B,\Delta}{C}}
\qquad
\infer{\tseq{\Gamma}{B/A}}{\tseq{\Gamma,A}{B}}
$$
while the \emph{right}-biased product $A\mulop B$ can be naturally combined with \emph{left} division $A\backslash B$.
$$
\infer{\tseq{\Gamma,\Delta,A\backslash B}{C}}{\tseq{\Gamma,B}{C} & \tseq{\Delta}{A}}
\qquad
\infer{\tseq{\Delta}{A\backslash B}}{\tseq{A,\Delta}{B}}
$$
To the best of my knowledge, these simple variations of Lambek's original rules have not appeared previously in the literature with a connection to the Tamari order.
Since circulating an earlier version of this paper \cite{Z2017tamari}, however, I learned from Jason Reed that he briefly considered precisely these rules in an unpublished note on ``queue logic'' \cite{Reed09queue}.
Reed did not remark the link with semi-associativity, but rather was originally interested in the apparent failure of queue logic to satisfy Nuel Belnap's \emph{display property}.

The sequent calculus can also be naturally extended with rules for a (left-biased) unit,
$$
\infer[IL]{\tseq{I,\Delta}{C}}{\tseq{\Delta}{C}} \qquad
\infer[IR]{\tseq{\cdot}{I}}{}
$$
although there are some subtleties involving focusing.
For one, the naive statement of the coherence theorem fails, as shown by the following example (due to Tarmo Uustalu) of two different focused derivations of the same sequent:
$$
\infer[\mulrule R]{\tseq{p,I,q}{(p\mul I)\mul q}}{
 \infer[\mulrule R]{\tseq{p,I}{p\mul I}}{\infer{\tseq{p}{p}}{} & \infer[IL]{\tseq{I}{I}}{\infer[IR]{\tseq{\cdot}{I}}{}}} &
 \infer{\tseq{q}{q}}{}}
\qquad
\infer[\mulrule R]{\tseq{p,I,q}{(p\mul I)\mul q}}{
 \infer[\mulrule R]{\tseq{p}{p\mul I}}{\infer{\tseq{p}{p}}{} & \infer[IR]{\tseq{\cdot}{I}}{}}
 &
 \infer[IL]{\tseq{I,q}{q}}{\infer{\tseq{q}{q}}{}}}
$$
However, this does not rule out a less naive formulation of coherence.

The name for Theorem~\ref{thm:coherence} is inspired by the terminology from category theory and Mac~Lane's coherence theorem for monoidal categories \cite{MacLane1963}.
Laplaza \cite{Laplaza1972} extended Mac~Lane's coherence theorem to the situation where there is no monoidal unit and the \emph{associator} $\alpha_{A,B,C} : (A \otimes B)\otimes C \to A \otimes (B \otimes C)$ is only a natural transformation rather than an isomorphism -- in other words, where the tensor product is only semi-associative.
Recently, Szlach\'anyi has isolated the notion of a \emph{skew-monoidal category} \cite{Szlachanyi2012}, which combines such a semi-associative tensor product together with a unit $I$ equipped with a pair of natural transformations $I\otimes A \to A$ and $A \to A \otimes I$ satisfying some axioms.
Skew-monoidal categories have been investigated actively in the past few years, including questions of coherence \cite{Uustalu2014} as well as their relation to the Tamari order \cite{LackStreet2014}.
Indeed, in recent work independent of ours, Bourke and Lack have arrived at a very closely related multicategorical analysis of skew-monoidal categories \cite{BourkeLack2017skewmulti}.
The reader familiar with the connection (pioneered by Lambek \cite{Lambek1969}) between multicategories and sequent calculus will recognize that the restricted rules $\mulrule L$ and $\mulrule R$ turn the sequent calculus into a ``left-representable'' multicategory, a natural variation of the well-known concept of a representable multicategory \cite{Hermida2000,LeinsterHOHC} considered in Bourke and Lack's paper \cite{BourkeLack2017skewmulti}.

Interestingly, Lambek, who originally presented his ``syntactic calculus'' as a tool for mathematical linguistics, also introduced a fully non-associative version \cite{Lambek1961} (cf.~\cite[Ch.~4]{MootRetoreLCG}).
One may wonder whether there are any linguistic motivations for semi-associativity, as an intermediate between these two extremes.

\section{A new proof for the lattice property of the Tamari order}
\label{sec:lattice}

In this section we explain how the focusing property can be used to give a new and more conceptual proof of Tamari's original observation that the posets $\Tam{n}$ of fully-bracketed words over a fixed sequence of $n+1$ letters ordered by a semi-associative law
are actually lattices.\footnote{This observation was already made in Tamari's thesis \cite{Tamari1951phd}, but the first published proof appeared in a paper with his student Haya Friedman \cite{FrTa67} (in a note at the end of the introduction to \cite{Tamari1964}, he mentions that this work was carried out by Friedman in 1958) and a simplified proof in a later paper with another student Samuel Huang \cite{HuTa72}.
For more recent proofs, see \cite{Knuth1993TamariTalk}, \cite[\S4]{MelliesART6} and \cite{CaspardSantocanaleWehrung2016}.}
We begin by extending the Tamari order to an ordering on \emph{contexts} (i.e., to lists of fully-bracketed words).
\begin{defi}
\label{defn:suborder}
The \defn{substitution order} on contexts is the least relation $\Gamma \le \Theta$ such that:
(1) $\tseq{\Gamma}{A}$ implies $\Gamma \le A$; (2) $\cdot \le \cdot$; and (3) if $\Gamma_1 \le \Gamma_2$ and $\Theta_1 \le \Theta_2$ then $(\Gamma_1,\Theta_1) \le (\Gamma_2,\Theta_2)$. 
\end{defi}
\noindent
The substitution order generated by the sequent calculus corresponds to the standard notion of the free monoidal category generated by a (posetal) multicategory \cite{LeinsterHOHC}, and is easily seen to be reflexive and transitive (inheriting these properties from $id$ and $cut$), and equipped with a contravariant action on derivable sequents (corresponding to a ``multi-cut'' rule).
\begin{prop}
\label{defn:multicut}
If $\Gamma_1 \le \Gamma_2$ and $\tseq{\Gamma_2}{A}$ then $\tseq{\Gamma_1}{A}$.
\end{prop}
\noindent
Now, let $\Frm$ stand for the poset of arbitrary formulas (= fully-bracketed words) under the Tamari order, and $\Ctx$ for the poset of arbitrary contexts under the generated substitution order.
We write $\Frm_\Omega$ and $\Ctx_\Omega$ to denote the restriction of these posets to a fixed frontier $\Omega$.
As already remarked (Prop.~\ref{prop:frontier-inv}), for any $\Omega$ of length $n+1$, $\Frm_\Omega$ is isomorphic to the poset $\Tam{n}$ of unlabelled binary trees with $n$ inner nodes under the right rotation order.
Similarly, for any $\Omega$ of length $n+1$, $\Ctx_\Omega$ is isomorphic to a poset of ``forests'' of unlabelled binary trees with $n+1$ leaves, which we denote by $\TamC{n}$ (note the different notation for the index, since we count inner nodes in trees but leaves in forests).

Our aim is to prove that all of the posets $\Frm_\Omega$ and $\Ctx_\Omega$ (and hence $\Tam{n}$ and $\TamC{n}$) are lattices.
The proof relies on two key observations:
\begin{enumerate}
\item The evident embedding $i : \Tam{n} \to \TamC{n}$ sending a binary tree with $n$ internal nodes to the singleton forest with $n+1$ leaves forms the right end of an \emph{adjoint triple,} with a left adjoint given by the left-associated product $\phi$, which in turn has a further left adjoint corresponding to the \emph{maximal decomposition} of a binary tree.
For completely general reasons, this allows us to reduce a join of trees to a join of forests.
\item Any forest with $n+1$ leaves divided among $k$ trees implicitly contains a \emph{composition} (i.e., ordered partition) of $n+1$ into $k$ parts, inducing a (surjective) monotone function $\alpha : \TamC{n} \to \Comp{n}$ from the poset of forests to the \emph{lattice of compositions ordered by refinement.}
In particular, this will allow us to reduce a join of forests in $\TamC{n}$ to a join of compositions in $\Comp{n} \cong 2^n$, together with joins of trees in $\Tam{m}$ for $m < n$.
\end{enumerate}
\begin{defi}
\label{defn:maxdecomp}
We say that an irreducible context $\Gamma^\irr$ is a \defn{maximal decomposition} of a formula $A$ if $\tseq{\Gamma^\irr}{A}$, and for any other $\Theta^\irr$, $\tseq{\Theta^\irr}{A}$ implies $\Theta^\irr \le \Gamma^\irr$.
\end{defi}
\begin{prop}
\label{prop:decompinv}
If $\Gamma^\irr$ is a maximal decomposition of $A$, then $\tseq{A,\Delta}{B}$ if and only if $\tseq{\Gamma^\irr,\Delta}{B}$.
\end{prop}
\begin{proof}
The forward direction is by cutting with $\tseq{\Gamma^\irr}{A}$, the backwards direction is by the deduction lemma (\ref{lem:focgenid}) and the universal property of $\Gamma^\irr$.
\end{proof}
\begin{prop}
\label{prop:decomp}
Let $\toC{A}$ be the irreducible context defined inductively by
$\toC{p} = p$ 
and
$\toC{A\mul B} = \toC{A},B$.
Then $\toC{A}$ is a maximal decomposition of $A$.
\end{prop}
\begin{proof}
We construct $\tseq{\toC{A}}{A}$ by induction on $A$, and check the universal property of $\toC{A}$ by induction on focused derivations of $\tseq{\Theta^\irr}{A}$.
\end{proof}
\begin{prop}
\label{prop:mudecompmu}
$\toF{\toC{A}} = A$ and $\toC{\toF{\Theta^\irr}} = \Theta^\irr$ for all $A$ and $\Theta^\irr$.
\end{prop}
\noindent
The maximal decomposition $\toC{A}$ of $A$ is essentially the same as what Chapoton \cite{Chapoton2006} calls the ``décomposition maximale'' of a binary tree, although note that the universal property expressed in Definition~\ref{defn:maxdecomp} is also familiar from studies of focusing in other settings (cf.~\cite{Z2008unityduality}).
\begin{prop}
\label{prop:adjoint3}
The operations $\toF{-}$ and $\toC{-}$ induce an \defn{adjoint triple}
\begin{center}
\begin{tikzcd}
 \Frm \arrow[hookrightarrow,"i"]{r} \arrow[bend left=75,"\toCsym"]{r}
  &
 \Ctx_+ \arrow[bend right]{l}[above]{\toFsym}
\end{tikzcd}
\end{center}
where
$\Ctx_+$ denotes the restriction of $\Ctx$ to non-empty contexts, and $i$ the evident embedding.
In other words, we have that
$$\toF{\Gamma} \le A\iff\Gamma \le i[A] \qquad\text{and}\qquad
\toC{A} \le \Theta\iff A \le \toF{\Theta}
$$
for all $\Gamma$, $A$, and $\Theta$.
\end{prop}
\begin{proof}
The first part
($\toF{\Gamma} \le A$ if and only if $\Gamma \le i[A]$ := $\tseq{\Gamma}{A}$)
amounts to soundness and completeness of the sequent calculus, while the second part
($\toC{A} \le \Theta$ if and only if $A \le \toF{\Theta}$)
follows from Propositions~\ref{prop:decompinv}--\ref{prop:mudecompmu}.
\end{proof}
\begin{rem}
Since the operations $i$, $\phi$, and $\psi$ are manifestly frontier-preserving, they restrict to adjoint triples between the posets $\Frm_\Omega$ and $\Ctx_\Omega$ of formulas and contexts over any fixed (non-empty) frontier $\Omega$.
In particular, they restrict to an adjoint triple between $\Tam{n}$ and $\TamC{n}$.
\end{rem}
\noindent
Let us write $A \vee_\Omega B$ (respectively, $\Gamma \sqcup_\Omega \Theta$) to denote the join of two formulas (respectively, contexts) with the same frontier $\Omega$, and $\bot_\Omega$ (respectively, $0_\Omega$) to denote the least formula (respectively, context) with frontier $\Omega$.
(Sometimes we will omit the index $\Omega$ when there is no ambiguity.)
\begin{lem}\label{lem:join-formula}
The following equations hold (whenever the corresponding joins exist):
\begin{align}
A \vee_\Omega B &= \toF{\toC{A} \sqcup_\Omega \toC{B}} \label{eqn:join-formula}\\
\bot_\Omega &= \toF{0_\Omega} \label{eqn:bot-formula}
\end{align}
\end{lem}
\begin{proof}
By Propositions~\ref{prop:mudecompmu} and \ref{prop:adjoint3}, since left adjoints preserve arbitrary joins.
\end{proof}
\noindent
Now, the existence of a least context is trivial to establish.
\begin{prop}
\label{prop:frontier-minimal}
If $\frontier{A} = \Omega$ then $\tseq{\Omega}{A}$.
\end{prop}
\begin{prop}
\label{prop:bot-context}
If $\frontier{\Theta} = \Omega$ then $\Omega \le \Theta$.
\end{prop}
\begin{proof}
By induction.
\end{proof}
\noindent
In other words, we have that
\begin{equation}
0_\Omega = \Omega. \label{eqn:bot-context}
\end{equation}
This in turn implies by \eqref{eqn:bot-formula} that the least formula is the ``left comb'' $\bot_\Omega = \toF{\Omega}$, which again is easy to verify.

The case of binary joins is more subtle, and as previewed above, we will make use of the fact that 
$\Ctx_\Omega$ lives naturally over the lattice $\CompC{\Omega}$ of compositions of $\Omega$ ordered by refinement.
Recall that a \emph{composition} of an integer $n$ (into $k$ parts) is a tuple of positive integers $\alpha = (n_1;\dots;n_k)$ such that $n = n_1 + \dots + n_k$ (or equivalently, a surjective monotone function $\mathbf{n} \to \mathbf{k}$ from the ordinal $\mathbf{n} = \set{1 < \dots < n}$ to the ordinal $\mathbf{k} = \set{1 < \dots < k}$).
Similarly, a composition of a list $\Omega$ is a list of non-empty blocks $\alpha = (\Omega_1; \dots; \Omega_k)$ such that $\Omega$ is the concatenation $\Omega_1,\dots,\Omega_k$.
Compositions form the morphisms of a category (isomorphic to the category of linearly ordered sets and surjective monotone functions), where a composition $\alpha = (\Omega_1; \dots; \Omega_k)$ can be \emph{composed} with a composition $\beta$ of its list of blocks $[\Omega_1],\dots,[\Omega_k]$, to obtain another composition $\beta\circ \alpha$ of $\Omega$ in which some consecutive blocks of $\alpha$ have been merged.
This in turn induces a partial order $\CompC{\Omega}$ on compositions of the same list $\Omega$: $\alpha$ is said to \emph{refine} another composition $\alpha'$ just in case $\alpha' = \beta\circ \alpha$ for some $\beta$.
\begin{prop}
The function 
$
(A_1,\dots,A_k) \mapsto (\frontier{A_1}; \dots; \frontier{A_k})
$
sending a context to the underlying composition of its frontier is monotone relative to the substitution ordering on contexts and the refinement ordering on compositions.
\end{prop}
\noindent
The lattice structure of $\CompC{\Omega}$ is easy to describe, with $\CompC{\Omega} \cong 2^n$ for any non-empty list $\Omega$ of length $n+1$.
(To see this one can think of the block notation for compositions as encoding $n$-ary bit vectors, with the comma `,' standing for 1 and the semicolon `;' for 0.)
But to compute the join of two contexts with the same frontier, we will also need the following more informative characterization of the join of compositions.
\begin{prop}
\label{prop:comp-pushout}
Any pair of compositions $\alpha$ and $\alpha'$ of $\Omega$ has a \defn{pushout} consisting of a pair $(\beta,\beta')$ satisfying $\beta \circ \alpha = \beta' \circ \alpha' = \alpha \sqcup \alpha'$, and such that for any other pair $(\gamma, \gamma')$ satisfying $\gamma \circ \alpha = \gamma' \circ \alpha'$, there exists a unique $\delta$ such that $\gamma = \delta \circ \beta$ and $\gamma' = \delta \circ \beta'$.
\end{prop}
\noindent
\begin{lem}\label{lem:join-context}
Let $\Gamma$ and $\Theta$ be two contexts with the same frontier $\frontier{\Gamma} = \frontier{\Theta} = \Omega$, and let $\alpha_\Gamma$ and $\alpha_\Theta$ be the corresponding compositions of $\Omega$.
Then
\begin{equation}\label{eqn:join-context}
\Gamma \sqcup \Theta = i[\toF{\Gamma_1} \vee \toF{\Theta_1}],\dots,i[\toF{\Gamma_k}\vee\toF{\Theta_k}]
\end{equation}
(assuming the corresponding joins exist), 
where the compositions $\beta_\Gamma = (\Gamma_1 ; \dots ; \Gamma_k)$ and $\beta_\Theta = (\Theta_1 ; \dots ; \Theta_k)$ of $\Gamma$ and $\Theta$ are obtained as the pushout of $\alpha_\Gamma$ and $\alpha_\Theta$.
\end{lem}
\noindent
The proof of Lemma~\ref{lem:join-context} can be reduced to the following simpler cases.
\begin{lem}\label{lem:join-pair}
If $\frontier{A} = \frontier{B}$ and $\frontier{\Gamma} = \frontier{\Theta}$ then
$
(A,\Gamma) \sqcup (B,\Theta)
=
(A \vee B),(\Gamma\sqcup \Theta)
$.
\end{lem}
\begin{proof}
We have to show that $A,\Gamma \le \Delta$ and $B,\Theta\le \Delta$ if and only if 
$(A\vee B), (\Gamma\sqcup\Theta) \le \Delta$.
The backward direction is immediate.
For the forward direction, by induction on $\Delta$ we can reduce to the case of a singleton context $\Delta = i[C]$.
So suppose that $\tseq{A,\Gamma}{C}$ and $\tseq{B,\Theta}{C}$.
By the properties of maximal decomposition, this implies that $\toC{A},\Gamma \le \toC{C}$ and $\toC{B},\Theta \le \toC{C}$.
But then we must have $\toC{C} = (\Delta,\Delta')$ for some $\Delta$ and $\Delta'$ such that $\toC{A} \le \Delta$ and $\toC{B} \le \Delta$ and $\Gamma \le \Delta'$ and $\Theta \le \Delta'$, 
which in turn implies that 
$\tseq{A \vee B,\Gamma \sqcup \Theta}{C}$.
\end{proof}
\begin{lem}\label{lem:join-relprime}
If $\Gamma$ and $\Theta$ are such that $\alpha_\Gamma \sqcup \alpha_\Theta$ is the one-block composition (i.e., $\alpha_\Gamma$ and $\alpha_\Theta$ are ``relatively prime''), then $\Gamma \sqcup \Theta = i[\toF{\Gamma} \vee \toF{\Theta}]$.
\end{lem}
\begin{proof}
We have to show that $\Gamma \le \Delta$ and $\Theta\le \Delta$ if and only if 
$i[\toF{\Gamma}\vee \toF{\Theta}] \le \Delta$.
Again the backward direction is immediate.
In the forward direction, if both $\Gamma \le \Delta$ and $\Theta \le \Delta$, then $\Delta$ is necessarily a singleton $\Delta = i[C]$, by assumption that $\alpha_\Gamma$ and $\alpha_\Theta$ are relatively prime.
But then $\toF{\Gamma} \le C$ and $\toF{\Theta} \le C$ from the adjunction $\phi \dashv i$, and hence 
$\toF{\Gamma}\vee \toF{\Theta} \le C$.
\end{proof}
\begin{proof}[Proof of Lemma~\ref{lem:join-context}.]
By repeated application of Lemmas~\ref{lem:join-pair} and \ref{lem:join-relprime}, relying on the properties of the 
compositions $\beta_\Gamma$ and $\beta_\Theta$ as the pushout of $\alpha_\Gamma$ and $\alpha_\Theta$.
\end{proof}

\begin{thm}
The posets $\Frm_\Omega$ and $\Ctx_\Omega$ are lattices, for all $\Omega$.
\end{thm}
\begin{proof}
We can use equations \eqref{eqn:join-formula}, \eqref{eqn:bot-formula}, \eqref{eqn:bot-context}, and \eqref{eqn:join-context} to recursively compute the join of any (finite) collection of formulas or contexts over the same frontier, with a base case of $p \vee p = p$ for atomic formulas.
This recursion is well-founded because the maximal decomposition of a formula is irreducible, and irreducible contexts over the same frontier always share an atom as their first formula (i.e., we have that $k \ge 2$ for the number of parts in \eqref{eqn:join-context}, whenever $\Gamma^\irr = \toC{A}$ and $\Delta^\irr = \toC{B}$ are the maximal decompositions of non-atomic formulas).
Finally, by a standard construction, existence of joins also implies the existence meets, since $\Frm_\Omega$ and $\Ctx_\Omega$ are finite.
(Alternatively, meets may be computed directly by dualizing the whole construction, starting from the sequent calculus for the dual Tamari order described in the second paragraph of Section~\ref{sec:seqcal:notes}.)
\end{proof}

\begin{exa}
Let $A = p\mul((q\mul(r\mul((s\mul t)\mul u)))\mul v)$ and $B = (p\mul (q\mul r))\mul ((s\mul t)\mul (u\mul v))$
be two fully-bracketed words.
We illustrate the recursive computation of $A \vee B$:
$$
\small
\begin{array}{c | c|c|c|c|c}
\text{round} & A & \toC{A} & B & \toC{B} & \toC{A} \sqcup \toC{B} \\
\hline
1 & p((q(r((st)u)))v) & p,(q(r((st)u)))v & (p(qr))((st)(uv)) & p,qr,(st)(uv) & p,A_2 \vee B_2\\ 
\hline
2 & (q(r((st)u)))v & q,r((st)u),v & (qr)((st)(uv)) & q,r,(st)(uv) & q,A_3 \vee B_3 \\
\hline
3 & (r((st)u))v & r, (st)u, v & r((st)(uv)) & r, (st)(uv) & r,A_4 \vee B_4 \\
\hline
4 & ((st)u)v & s,t,u,v & (st)(uv) & s,t,uv & s,t,A_5 \vee B_5 \\
\hline
5 & uv & u,v & uv & u,v & u,v
\end{array}
$$
We conclude that $A \vee B = (p\mul (q\mul (r\mul ((s\mul t)\mul (u\mul v)))))$.
\end{exa}

\section{Counting intervals in Tamari lattices}
\label{sec:enumeration}

In this section we explain how the coherence theorem can be used to derive Chapoton's result (mentioned in Section~\ref{sec:intro:chapo}) that the number of intervals in $\Tam{n}$ is given by Tutte's formula \eqref{tutte-formula1}.
We assume some basic familiarity with generating functions \cite{WilfGF}.

To ``count intervals'' means to compute the cardinality of the set
$$
\Int{n} = \set{(A,B) \in \Tam{n} \times \Tam{n} \mid A \le B}
$$
as a function of $n$.
By the coherence theorem, the problem of counting intervals can be reduced to the problem of \emph{counting focused derivations,} and since the latter are just special kinds of trees, this problem lends itself readily to being solved using generating functions.

Consider the formal power series $L(z,x)$ and $R(z,x)$ defined by
$$L(z,x) = \sum_{n,k} \ell_{n,k} z^n x^k \quad \text{and}\quad R(z,x) = \sum_{n,k} r_{n,k} z^n x^k$$
where $\ell_{n,k}$ (respectively, $r_{n,k}$) is the number of focused derivations of sequents whose right-hand side is a formula of size $n$ (over some arbitrary frontier) and whose left-hand side is a context (respectively, irreducible context) of length $k$.
We write $L_1(z)$ to denote the coefficient of $x^1$ in $L(z,x)$.
\begin{prop}
$L_1(z)$ is the ordinary generating function counting intervals in $\Tam{n}$.
\end{prop}
\begin{proof}
The coefficients $\ell_{n,1}$ give the number of focused derivations of sequents of the form $\fseq{A}{B}$, where $\sz{A} = \sz{B} = n$, so $\ell_{n,1} = |\Int{n}|$
by Theorem~\ref{thm:coherence}.
\end{proof}
\begin{prop}\label{prop:genfunc}
$L$ and $R$ satisfy the equations:
\begin{align}
L(z,x) &= x^{-1}(L(z,x) - x L_1(z)) + R(z,x) \label{eqn:L}\\
R(z,x) &= z R(z,x)L(z,x) + x\label{eqn:R}
\end{align}
\end{prop}
\begin{proof}
The equations are derived directly from the inductive structure of focused derivations:
\begin{itemize}
\item The first summand in \eqref{eqn:L} corresponds to the contribution from the $\mulrule L$ rule, which transforms any $\fseq{A,B,\Gamma}{C}$ into $\fseq{A\mul B,\Gamma}{C}$.
The context in the premise must have length $\ge 2$ which is why we subtract the $xL_1(z)$ factor, and the context in the conclusion is one formula shorter which is why we divide by $x$.
The second summand is the contribution from irreducible contexts.
\item
The first summand in \eqref{eqn:R} corresponds to the contribution from the $\mulrule R^\foc$ rule, which transforms $\fseq{\Gamma^\irr}{A}$ and $\fseq{\Delta}{B}$ into $\fseq{\Gamma^\irr,\Delta}{A\mul B}$: the length of the context in the conclusion is the sum of the lengths of $\Gamma^\irr$ and $\Delta$, while the size of $A\mul B$ is one plus the sum of the sizes of $A$ and $B$, which is why we multiply $R$ and $L$ together and then by an extra factor of $z$.
The second summand is the contribution from $id^\atm : \fseq{p}{p}$.
\qedhere
\end{itemize}
\end{proof}
\begin{prop}
\label{prop:L1R1}
$L_1(z) = R(z,1)$.
\end{prop}
\begin{proof}
This follows algebraically from \eqref{eqn:L}, but we can interpret it constructively as well.
The coefficient of $z^n$ in $R(z,1)$ is the formal sum $\sum_k r_{n,k}$, giving the number of focused derivations of sequents whose right-hand side is a formula of size $n$ and whose left-hand side is an irreducible context of arbitrary length.
But by the existence of maximal decompositions (Props.~\ref{prop:decompinv}--\ref{prop:mudecompmu}), we know that there is a 1-to-1 correspondence between derivable sequents of the form $\tseq{\Gamma^\irr}{B}$ and ones of the form $\tseq{A}{B}$.
\end{proof}
\noindent
After substituting $L_1(z) = R(z,1)$ into \eqref{eqn:L} and applying a bit of algebra, we obtain another formula for $L$ in terms of a ``discrete difference operator'' acting on $R$:
\begin{equation}
L(z,x) = x\frac{R(z,x) - R(z,1)}{x-1} \label{eqn:L'}
\end{equation}
The recursive (or ``functional'') equations \eqref{eqn:R} and \eqref{eqn:L'} can be easily unrolled using computer algebra software to compute the first few dozen coefficients of $R$ and $L$:
\small
\begin{align*}
 R(z,x) &
 = x + x^2z + (x^2+2x^3)z^2 + (3x^2 + 5x^3 + 5x^4)z^3 + (13x^2 + 20x^3 + 21x^4 + 14x^5)z^4 + \dots\\
L_1(z) &
 = R(z,1) = 1 + z + 3z^2 + 13z^3 + 68z^4 + 399z^5 + 2530z^6 + 16965 z^7 + \dots
\end{align*}
\normalsize
\begin{thm}[Chapoton \cite{Chapoton2006}]
\label{thm:chapo}
$|\Int{n}| = \frac{2(4n+1)!}{(n+1)!(3n+2)!}$.
\end{thm}
\begin{proof}
At this point we simply appeal to a paper of Cori and Schaeffer \cite{CoSch2003}, because equations \eqref{eqn:R} and \eqref{eqn:L'} are a special case of the functional equations given there for the generating functions of \emph{description trees of type $(a,b)$}, where $a = b = 1$.
Cori and Schaeffer explain (see \cite[pp.~171--174]{CoSch2003})
how to solve these equations abstractly for $R(z,1)$ using Brown and Tutte's ``quadratic method'', and then how to derive the explicit formula above in the specific case that $a = b = 1$ via Lagrange inversion (essentially as the formula was originally derived by Tutte for planar triangulations).
\end{proof}
\noindent
Let's take a moment to discuss Chapoton's original proof of Theorem~\ref{thm:chapo}, which despite being obtained in a different way, we should emphasize has many commonalities with the one given here via the coherence theorem.
Chapoton likewise defines a two-variable generating function $\Phi(z,x)$ enumerating intervals in the Tamari lattices $\Tam{n}$, where the parameter $z$ keeps track of $n$, and the parameter $x$ keeps track of the \emph{number of segments along the left border} of the tree at the lower end of the interval.\footnote{Or rather at the upper end, since Chapoton uses the dual ordering convention (cf.~Footnote~\ref{foot1}).}
By an extended combinatorial analysis, Chapoton derives the following equation for $\Phi$:
\begin{equation}
\Phi(z,x) = x^2 z(1+\Phi(z,x)/x)\left(1 + \frac{\Phi(z,x) - \Phi(z,1)}{x-1}\right) \label{eqn:chapofun}
\end{equation}
He manipulates this equation and eventually solves for $\Phi(z,1)$ as the root of a certain polynomial, from which he derives Tutte's formula \eqref{tutte-formula1}, again by appeal to a different result in the paper by Cori and Schaeffer \cite{CoSch2003}.

If we give a bit of thought to these definitions, it is easy to see that the number of segments along the left border of a tree $A$ is equal to the number of formulas in its maximal decomposition $\toC{A}$ -- meaning that the generating function $\Phi(z,x)$ apparently contains exactly the same information as $R(z,x)$!
There is a small technicality, however, due to the fact that Chapoton only considers the $\Tam{n}$ for $n\ge 1$.
In fact, the two generating functions are related by a small offset (corresponding to the coefficient of $z^0$ in $R(z,x)$):
$
\Phi(z,x) = R(z,x) - x 
$.
Indeed, it can be readily verified that equation \eqref{eqn:chapofun} follows from \eqref{eqn:R} and \eqref{eqn:L'}, applying the substitution $R(z,x) = x + \Phi(z,x)$.

\section*{Acknowledgment}

I thank Wenjie Fang, Paul-André Melliès, Jason Reed, Gabriel Scherer, and Tarmo Uustalu for interesting conversations about this work, as well as the anonymous FSCD and LMCS reviewers for some helpful comments.
The visualization of $\Tam{4}$ in Figure~\ref{fig:tam3tam4} was created using the help of Brent Yorgey's Diagrams library for Haskell. 

\bibliographystyle{plainurl}
\bibliography{tamari-lmcs}

\onecolumn

\appendix
\section{A missed connection}
\label{sec:postscript}

Tamari obtained his doctorate in 1951 from the University of Paris.\footnote{A fascinating account of Tamari's eventful life and scientific contributions appears in the first chapter of \cite{TamariFestschrift}, see ``Dov Tamari (formerly Bernhard Teitler)'' by Folkert M\"uller-Hoissen and Hans-Otto Walther.
That article is also the source from which I learned of Tamari's missed connection with Lambek (see Footnote 45 on page 18).}
His thesis \cite{Tamari1951phd} revisited the important problem in algebra (first solved by Malcev) of characterizing the conditions under which a semigroup is embeddable into a group, which Tamari showed can be reduced to the more elementary problem of determining when a partial binary operation\footnote{Somewhat disorientingly for modern readers such as the present author, in his thesis and early work Tamari referred to the algebraic structures consisting of sets equipped with a partial binary operation as ``monoids''. (In later work \cite{Tamari1982graphic} he referred to them as ``bins''.)} is fully associative (that is, when all possible bracketings of a word have the same value).
(The thesis includes a description of the partial orders $\Tam{n}$ with the claim that they are lattices, as well as graphical depictions of $\Tam{3}$ and $\Tam{4}$ as convex polytopes, although none of this made it into the abridged version published in BSMF \cite{Tamari1951phd}.)

Coincidentally, the embedding problem for semigroups was also studied around the same time by Lambek, who arrived at some related ideas \cite{Lambek1951}.
Well, it turns out that the two almost began a collaboration!
As Tamari recounts in \cite[p.~6]{Tamari1962/63}:\footnote{``In 1951, after his thesis \cite{Tamari1951phd} and after the publication of \cite{Lambek1951}, the author proposed to Lambek joint work, to highlight the preponderant role played by general associativity. Unfortunately, due to external circumstances, this work was never written.'' (Own translation, citation numbers realigned with present bibliography.)}
\begin{quote}
En 1951, après sa thèse \cite{Tamari1951phd} et après la publication de \cite{Lambek1951},
l'auteur a proposé à LAMBEK un travail commun, pour mettre en évidence le rôle prépondérant joué par l'associativité générale.
Malheureusement, par suite de circonstances extérieures, ce travail n'a jamais été écrit.
\end{quote}
Disappointingly, Tamari does not elaborate any further.
But one is certainly lead to fantasize about what were the ``circonstances extérieures'' which prevented Lambek and Tamari from pursuing this collaboration, and what might have been the outcome if they had.
It seems quite likely that between Lambek's introduction of the associative and non-associative syntactic calculi, and Tamari's continued study of the lattices induced by a semi-associative law, the pair would have eventually noticed the simple sequent calculus which we have studied here.
What else might they have noticed?

\end{document}